\documentclass[reqno]{amsart}
\usepackage{csquotes}
\usepackage{hyperref}
\usepackage{amsmath}   
\usepackage{amssymb} 
\usepackage{mathrsfs} 
\usepackage{accents}
\usepackage{graphicx}

\usepackage{tabularx,colortbl}

\usepackage{tikz} 
\usetikzlibrary{calc} 
\usepackage{xcolor}  
\usetikzlibrary{arrows.meta} 
\usepackage{amsthm} 
\usepackage{float}
\usepackage{enumitem}
\usepackage[english]{babel}
\usepackage[font=footnotesize, labelfont=bf]{ caption}
\usepackage{ esint }
\usepackage{ dsfont }

\newcommand{\eg}{{\it e.g.}}
\newcommand{\ie}{{\it i.e.}}

\theoremstyle{plain}
\begingroup
\newtheorem{theorem}{Theorem}[section]
\newtheorem{lemma}[theorem]{Lemma}
\newtheorem{proposition}[theorem]{Proposition}

\endgroup

\theoremstyle{definition}
\begingroup
\newtheorem{definition}[theorem]{Definition}

\endgroup
\theoremstyle{remark}
\newtheorem{remark}[theorem]{Remark}
\renewcommand{\tilde}{\widetilde}
\newcommand{\sm}{\setminus}

\renewcommand{\d}{ \mathrm{d}}

\DeclareMathOperator{\supp}{supp}

\newcommand{\pvint}{\, \mathrm{p.v.} \! \int}
\newcommand{\Schwartz}{\mathscr{S}}
\newcommand{\D}{\mathrm{D}}
\newcommand{\Dcal}{\mathcal{D}}

\numberwithin{equation}{section}
\newcommand{\N}{\mathbb{N}}

\newcommand{\R}{\mathbb{R}}
\newcommand{\C}{\mathbb{C}}
\renewcommand{\S}{\mathbb{S}} 
\renewcommand{\H}{\mathcal{H}} 
\newcommand{\de}{\partial} 
\newcommand{\e}{\varepsilon}

\newcommand{\x}{{\times}}

\renewcommand{\L}{\mathcal{L}}
\renewcommand{\hat}{\widehat}
\newcommand{\W}{\mathcal{W}}
\newcommand{\F}{\mathscr{F}}

\usepackage{mathtools}
\mathtoolsset{showonlyrefs}  

\newcommand{\myequation}{\begin{equation}}
  \newcommand{\myendequation}{\end{equation}}
  \let\[\myequation
  \let\]\myendequation

  \newenvironment{acknowledgements}{%
  \begin{abstract}
}{%
  \end{abstract}
}

  \title[Comparison between peridynamics and wave equation]{Comparison between solutions \\ to the linear peridynamics model \\ and solutions to the classical wave equation}

  \author[G. M. Coclite]{G. M. Coclite}
  \address[Giuseppe Maria Coclite]{\newline
     Dipartimento di Meccanica, Matematica e Management, Politecnico di Bari,
    Via E.~Orabona 4, I--70125 Bari, Italy.}
  \email[]{giuseppemaria.coclite@poliba.it}
  
  \author[S. Dipierro]{S. Dipierro}
  \address[Serena Dipierro]{\newline Department of Mathematics and Statistics,
  University of Western Australia, 35 Stirling Highway, WA6009 Crawley, Australia}
  \email[]{serena.dipierro@uwa.edu.au}

  \author[F. Maddalena]{F. Maddalena}
  \address[Francesco Maddalena]{\newline
   Dipartimento di Meccanica, Matematica e Management, Politecnico di Bari,
   Via E.~Orabona 4, I--70125 Bari, Italy.}
  \email[]{francesco.maddalena@poliba.it}
  
  \author[G. Orlando]{G. Orlando}
  \address[Gianluca Orlando]{\newline
   Dipartimento di Meccanica, Matematica e Management, Politecnico di Bari,
   Via E.~Orabona 4, I--70125 Bari, Italy.}
  \email[]{gianluca.orlando@poliba.it}
  
  \author[E. Valdinoci]{E. Valdinoci}
  \address[Enrico Valdinoci]{\newline Department of Mathematics and Statistics,
  University of Western Australia, 35 Stirling Highway, WA6009 Crawley, Australia.}
  \email[]{enrico.valdinoci@uwa.edu.au}

\usepackage[
backend=biber,
style=numeric-verb,
sorting=nty,
maxbibnames=99
]{biblatex}
\begin{document}

\begin{abstract} In this paper, we consider an equation inspired by
linear peridynamics and we establish its connection with the classical wave equation.

In particular, given a horizon~$\delta>0$ accounting for the region of influence around a material point,
we prove existence and uniqueness of a solution~$u_\delta$ and demonstrate the convergence of~$u_\delta$ to solutions to the classical wave equation as~$\delta \to 0$. 

Moreover, we prove that the solutions to the peridynamics model with small frequency initial data
are close to solutions to the classical wave equation. 
\end{abstract}

\maketitle

\setcounter{tocdepth}{1}
\tableofcontents

\section{Introduction}

The theory of Continuum Mechanics lays its foundation on the notion of internal forces acting in a material body~\cite{Gur73}.
Classically, a locality assumption allows one to condense the description of the internal forces via the Cauchy stress tensor. 
In this setting, it is well known that the local balance of momentum is given by a system of partial differential equations, whose prototype is the classical linear wave equation for~$\R^d$-valued functions.
However, modern problems in Materials Science require to broaden one's perspective in order to treat more complex material behaviors (fracture, plasticity, design of smart materials) that go beyond the classical theory ruled by PDEs.
One of the promising routes is to explore the possibility of removing the locality assumption to allow for a more general scenario. 
In this context, the theory of peridynamics proposed by Stewart Silling~\cite{Sil00} moves its steps on this route. 
See~\cite{SilEptWeck07, SilLeh10, Sil10} for further developments in the field and~\cite{LopPel21, LopPel22, LopPel22-2, CocFanLopMadPel20, CocCocMadPol24} for numerical methods for the peridynamics model.

The aim of this paper is to analyze in detail the mathematical relationship between the peridynamics model and the classical wave equation.

% A classical problem in material science consists in predicting the behavior of materials under deformation, damage, and fracture. For this purpose, it is often necessary to introduce mathematical theories going beyond the realm of partial differential equations: indeed,
% differently than continuum mechanics, realistic solutions of these problems are expected to
% show discontinuities and cracks and partial differential equations are typically
% not well-suited to describe these features, due to their use
% of continuous fields and derivatives. 

% Thus, as proposed by Stewart Silling in~\cite{Sil00},
% to handle complex material behavior like fracture mechanics,
% peridynamics should go beyond the classical elasticity theory and rely on integral equations that do not require continuity assumptions.
 
In the description of materials provided by these nonlocal equations,
the size of the interaction range within the material body, called the ``peridynamic horizon'', plays a crucial role. 
Roughly speaking, the horizon, which will be denoted by~$\delta$ in this paper, is the radius of a spherical neighborhood within which a point interacts with other points in the material. 
The first goal of this paper is to understand how the size of the horizon influences the model by rigorously detecting the asymptotic behavior of the solutions as~$\delta\to0$ (Theorem~\ref{thm:delta_to_0} below). 

The problem of the asymptotic behavior of the peridynamics model with vanishing horizon has been previously addressed in the literature in different contexts. In the static case, this has been investigated, \eg, in~\cite{SilLeh08, DuGunLehZho13} and in~\cite{MenDu15, BelMorPed15,BelCueMor20} via~$\Gamma$-convergence tools. 
In the dynamic case, \cite{EmmWeck07} prove a local convergence result for evolutions inside a bounded domain. 
Our result differs from the previous approaches, in that it stems from an intrinsic property of the peridynamics evolution model, that is one of the features which makes it substantially different from the classical wave equation. This property can be seen as a nontrivial dispersion relation, for which different frequencies propagate at different speeds~\cite{WeckSillAsk08,CocDipFanMadRomVal21, CocDipFanMadVal22, CocCocFanMad23}.  

% \begin{itemize}
%   \item \cite{EmmWeck07} local limit of evolutions inside a bounded domain, linear case 
%   \item \cite{Sil08} convergence of stresses 
%   \item \cite{DuGunLehZho13} convergence of equilibria
%   \item  \cite{MenDu15} $\Gamma$-convergence, static case, linear
%   \item \cite{BelMorPed15,BelCueMor20} $\Gamma$-convergence, static case, nonlinear
% \end{itemize}

We pursue this line of thought and move to the second goal of this paper. The presence of the nontrivial dispersion relation suggests that another meaningful parameter is the size~$R$ of the range of frequencies of the initial data. 
We will show that, for initial data with frequencies supported in a small region, the solutions to the peridynamics model are close to solutions to the classical wave equation.
Specifically, we estimate the Sobolev norm of the gap between the solutions of the aforementioned models in terms of~$\delta$ and~$R$ (Theorem~\ref{thm:low_frequencies} below). 

We describe now more precisely the model studied in this paper and the results obtained. 
\medskip

% the ``horizon'' is a key concept that defines the range of interaction between particles. Roughly speaking, the horizon, which will be denoted by~$\delta$ in this paper, is the radius of a spherical neighborhood within which a point interacts with other points in the material.
% One of the main goal of this paper is to understand how the size of the horizon influences the model by rigorously detecting the asymptotic behavior of the solutions as~$\delta\to0$.

The theory of peridynamics is quite broad and comprises different formulations. For the analysis from a mathematical perspective of general nonlinear equations in this field, see~\cite{CocDipMadVal18}. 
%, accounting for large deformations, plasticity, material failure,crack propagation, and damage accumulation. 
In this paper, we will focus on the linear theory,
in which  the internal force interactions behave linearly with respect to displacement increments.
% This simplification is standard in elasticity and essentially
% assumes small deformations and strains. 
We will also consider the problem in
the whole space~$\R^d$, in any dimension~$d \geq 1$.

More specifically, we are interested in solutions~$u \colon (0,+\infty) \x \R^d \to \R^d$ to the equation 
    \[
\begin{cases} \label{eq:linear_peridynamics}
  \de_{tt} u(t,x) - K_\delta[u(t,\cdot)](x) = 0 \, , &   (t,x) \in (0,+\infty) \x \R^d \, , \\
  u(0,x) = u_0(x) \, , \quad \de_t u(0,x) = v_0(x) \, , &  x \in \R^d \,,
\end{cases} 
\]
where~$K_\delta[u]$ is the linear (isotropic) peridynamics operator given by
\[ \label{eq:def_K_delta}
K_\delta[u](x) := - \frac{\kappa}{\delta^{2(1-\alpha)}}\pvint_{\R^d} \chi_\delta(y) \frac{u(x) - u(x-y)}{|y|^{d + 2 \alpha}} \, \d y \, , \quad \text{for every } x \in \R^d \, .
\] 
The function~$\chi_\delta$ is supported in the ball of radius~$\delta$, where~$\delta > 0$ is the peridynamics horizon. 
The quantities~$\kappa \in( 0,+\infty)$ and~$\alpha \in (0,1)$ are constitutive parameters of the model and the notation ``$\mathrm{p.v.}$'' stands for the principal value of the integral.
Further details on this operator will be given in Section~\ref{SEC:OPE}. 

The drop of the locality assumption mentioned at the beginning is reflected by the introduction of the nonlocal operator~$K_\delta$ in~\eqref{eq:linear_peridynamics} in place of the Laplace operator, which instead rules the classical wave equation. 

Existence and uniqueness for solutions belonging to the energy space for~\eqref{eq:linear_peridynamics} (and its nonlinear versions) have been proved in~\cite{CocDipMadVal18}. Here we formulate well-posedness for~\eqref{eq:linear_peridynamics} in a more general functional setting in which we consider solutions in Sobolev spaces~$H^s$ with any real order of differentiation~$s\in\R$
(the technical details needed to treat these spaces will be recalled in Section~\ref{0djofsvnk}).
In this framework, we prove the following existence result, based on the dispersion relation~$\xi \mapsto \omega_\delta(\xi)$ of the peridynamics equation, depicted in Figure~\ref{fig:dispersion_relation} and studied in detail in Section~\ref{sec:peridynamics_operator}. 

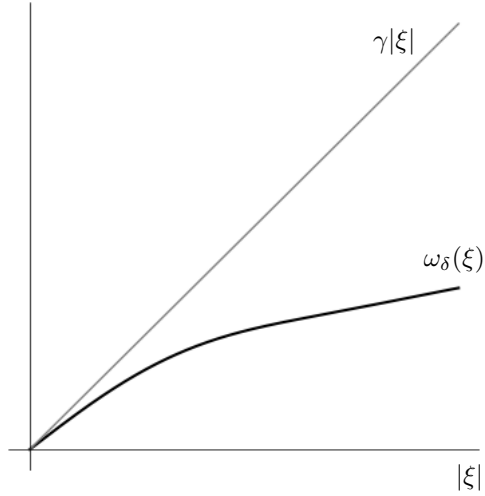
\begin{figure}[H]
  \begin{tikzpicture}
    \node (graph) at (0,0) {\includegraphics[scale=0.7]{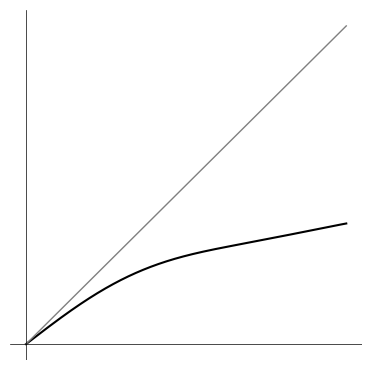}};
    \draw (3,-3.2) node {$|\xi|$};
    \draw (2.8,-0.3) node {$\omega_\delta(\xi)$};
    \draw (2,2.6) node {$\gamma |\xi|$};
  \end{tikzpicture}
  \caption{Aspect of the dispersion relation~$\omega_\delta(\xi)$ and comparison with the dispersion relation~$\gamma |\xi|$ of the wave equation with speed of propagation~$\gamma$.}  
  \label{fig:dispersion_relation}
\end{figure}

\begin{theorem} \label{thm:existence}
  Let~$s \in \R$. 
  Let~$u_0 \in H^s(\R^d;\R^d)$ and~$v_0 \in H^{s-\alpha}(\R^d;\R^d)$. 
  
  Then, there exists~$u_\delta \in C([0,+\infty);H^s(\R^d;\R^d)) \cap C^1((0,+\infty);H^{s-\alpha}(\R^d;\R^d))$ such that, for every~$t \in [0,+\infty)$, 
  \[ \label{eq:Fourier_solution}
    \hat{u}_\delta (t,\cdot) = \F[u_\delta(t,\cdot)]  = \cos(\omega_\delta(\cdot) t) \hat{u}_0  + \frac{\sin(\omega_\delta(\cdot) t)}{\omega_\delta(\cdot)} \hat{v}_0 \, .
  \]
  Moreover, $u_\delta$ is a distributional solution to the linear peridynamics model in~\eqref{eq:linear_peridynamics} with initial data~$(u_0,v_0)$.
\end{theorem} 

As customary, in~\eqref{eq:Fourier_solution}, the notations~$ \hat{u}_\delta $ and~$ \F[u_\delta]$
both stand for the Fourier transform of~$u_\delta$ 
with respect to the~$x$-variable (the two notations will be used interchangeably
depending on typographical convenience and the Fourier transform setting employed in this
paper will be recalled in Section~\ref{ihfgbvk9ikF}).
We also stress that, for simplicity, equation~\eqref{eq:linear_peridynamics} is written in the classical sense, but the solution found in Theorem~\ref{thm:existence} are intended in the sense of
distributions (the precise setting will be recalled in
Definition~\ref{def:distributional_solution}). This type of generality is important to
deal with discontinuous solutions, as required by a realistic model of peridynamics.

As a counterpart of the existence result in Theorem~\ref{thm:existence}, we also have a
uniqueness result, which goes as follows:

\begin{theorem} \label{thm:uniqueness}
  Let~$s \in \R$. 
  Let~$u_\delta \in L^1_{\mathrm{loc}}((0,+\infty);H^s(\R^d;\R^d))$ be a distributional solution to the linear peridynamics model in~\eqref{eq:linear_peridynamics} with initial data~$(0,0)$. 
Then, $u_\delta \equiv 0$.
\end{theorem}

We are now in a position to state precisely the comparison results contained in this paper. 
The first result is Theorem~\ref{thm:delta_to_0} below.
It gives a new perspective on the wave equation, which can be viewed as a limit case of the more general nonlocal peridynamics model as the horizon~$\delta$ goes to zero, \ie, as the scale of the internal interactions is small compared to the scale of the material body.

In the next statement, the speed of propagation of waves~$\gamma$ turns out to be the slope at the origin of the peridynamic  dispersion relation~$\omega_\delta$ in the radial direction, see Figure~\ref{fig:dispersion_relation} and formula~\eqref{eq:omega_delta_low_frequencies} below.
 
% To appreciate the role played by the horizon~$\delta$ in peridynamics, we study the
% asymptotics of the (unique) solution~$u_\delta$ constructed via Theorems~\ref{thm:existence} and~\ref{thm:uniqueness} and we relate it to the classical wave equation, according to the following result:

\begin{theorem} \label{thm:delta_to_0}
  Let~$s \in \R$. 
  Let~$u_0 \in H^s(\R^d;\R^d)$ and~$v_0 \in H^{s-\alpha}(\R^d;\R^d)$. 
  
Let~$u_\delta \in C([0,+\infty);H^s(\R^d;\R^d)) \cap C^1((0,+\infty);H^{s-\alpha}(\R^d;\R^d))$ be the unique distributional solution to the linear peridynamics model~\eqref{eq:linear_peridynamics} with initial data~$(u_0,v_0)$ provided by Theorems~\ref{thm:existence} and~\ref{thm:uniqueness}. 
  
  Let~$u \in C([0,+\infty);H^{s}(\R^d;\R^d)) \cap C^1((0,+\infty);H^{s-1}(\R^d;\R^d))$ be the unique distributional solution to the wave equation
\[
\begin{cases} \label{eq:wave}
  \de_{tt} u(t,x) - \gamma^2 \Delta u(x) = 0 \, , &   (t,x) \in (0,+\infty) \x \R^d \, , \\
  u(0,x) = u_0(x) \, , \quad \de_t u(0,x) = v_0(x) \, , &  x \in \R^d \, .
\end{cases} 
\]
  
  Then, for every~$T \in [0,+\infty)$, 
  \[
  \lim_{\delta\to0}
  \sup_{t \in [0,T]} \Big( \| u_\delta(t,\cdot) - u(t,\cdot) \|_{H^s(\R^d;\R^d)} + \| \de_t u_\delta(t,\cdot) - \de_t u(t,\cdot) \|_{H^{s-1}(\R^d;\R^d)} \Big) =0\, .
  \]
\end{theorem} 

We underline the importance of taking the initial data in the correct spaces, having a regularity gap of order~$\alpha$ between displacement and velocity. We comment more in detail about this caveat in Remark~\ref{rem:initial_data}. 

The second result establishes that the solutions to the peridynamics model with initial data having frequencies supported in a ball of radius~$R$ are close to solutions to the classical wave equation. 
The gap in Sobolev norm between the two solutions in a time interval~$[0,T]$ is (for a given~$L^2$ norm of the initial data) of\footnote{This estimate is meaningful for small~$T \delta R^2$. If~$T \delta R^2$ is large, the~$H^s$ norm on the left-hand side of~\eqref{star456789} is bounded uniformly in~$T$ and the estimate is trivial.} order~$T \delta R^2$.
This yields a quantitative criterion to compare the two models.

\begin{theorem} \label{thm:low_frequencies}
 Let~$\delta > 0$ and~$R \in (0,1)$. Let~$u_0$ and~$v_0$ be such that $\supp(\hat u_0) \subset \{|\xi| \leq R\}$ and $\supp(\hat v_0) \subset \{|\xi| \leq R\}$.
 
 Let~$u_\delta$ be the unique distributional solution to the linear peridynamics model~\eqref{eq:linear_peridynamics} with initial data~$(u_0,v_0)$ provided by Theorems~\ref{thm:existence} and~\ref{thm:uniqueness}. 
 
 Let~$u$ be the unique distributional solution to the wave equation~\eqref{eq:wave} with initial data~$(u_0,v_0)$. 
 
 Then, for every~$s \geq 0$ and for every~$T \in [0,+\infty)$,
  \begin{equation}\label{star456789}
  \sup_{t \in [0,T]}   \| u_\delta(t,\cdot) - u(t,\cdot) \|_{H^s(\R^d;\R^d)} \leq C T \delta R^2 \big( \|u_0\|_{L^2(\R^d;\R^d)} + \|v_0\|_{L^2(\R^d;\R^d)} \big)  \, ,
  \end{equation}
  and, for every~$s \geq 1$,
  \[
  \sup_{t \in [0,T]}  \| \de_t u_\delta(t,\cdot) - \de_t u(t,\cdot) \|_{H^{s-1}(\R^d;\R^d)} \leq  C T \delta R^2 \big( \|u_0\|_{L^2(\R^d;\R^d)} + \|v_0\|_{L^2(\R^d;\R^d)} \big)  \, .
  \]
\end{theorem}

To conclude our comparison between the two models, we analyze the energy
of the solution. Let
\[
E^{\gamma^2\Delta}[u](t) := \frac{1}{2} \int_{\R^d} |\de_t u(t,x)|^2 \, \d x + \frac{\gamma^2}{2} \int_{\R^d} |\nabla u(t,x)|^2 \, \d x  
\]
denote the total energy conserved by solutions to the wave equation~\eqref{eq:wave}.
Obviously, the energy~$E^{\gamma^2\Delta}$ is not conserved by solutions to the peridynamics model~\eqref{eq:linear_peridynamics}.
Nevertheless, in some sense, the energy~$E^{\gamma^2\Delta}$ is almost conserved by solutions to the peridynamics model, as stated precisely in the following result.

\begin{proposition}\label{prop:almost energy conservation}
  Let~$\e > 0$. Then there exists~$\eta_0 > 0$ such that for~$\delta R < \eta_0$ the following holds true.

  Let~$u_0$ and~$v_0$ be such that $\supp(\hat u_0) \subset \{|\xi| \leq R\}$ and $\supp(\hat v_0) \subset \{|\xi| \leq R\}$. 
  Let~$u_\delta$ be the unique distributional solution to the linear peridynamics model~\eqref{eq:linear_peridynamics} with initial data~$(u_0,v_0)$ provided by Theorems~\ref{thm:existence} and~\ref{thm:uniqueness}.
  
  Then, for every~$t \in [0,+\infty)$, we have that
  \[
  |E^{\gamma^2\Delta}[u_\delta](t) - E^{\gamma^2\Delta}[u_\delta](0)| \leq \e E^{\gamma^2\Delta}[u_\delta](0) \, .
  \]
\end{proposition}

The paper is organized as follows. In Section~\ref{yturieowghfdj6758493}
we present the basic notation and the functional setting used
throughout the manuscript. 
Section~\ref{sec:peridynamics_operator} is devoted to the peridynamics
operator and its properties.
In Section~\ref{67r483ebcnx98765yhbnetuku} we provide some well-posedness results
for the linear peridynamics model
and present the proofs of Theorems~\ref{thm:existence} and~\ref{thm:uniqueness}. 
Section~\ref{deltasmall76859403} contains the proof of
Theorem~\ref{thm:delta_to_0}, and Section~\ref{5647382prooft362y584} the proofs of Theorem~\ref{thm:low_frequencies} and
Proposition~\ref{prop:almost energy conservation}.

\section{Basic notation and functional setting}\label{yturieowghfdj6758493}

\subsection{Test functions and distributions}
In this paper, the~$d$-dimensional Lebesgue measure in~$\R^d$ is denoted by~$\L^d$ and the~$(d-1)$-dimensional Hausdorff measure is denoted by~$\H^{d-1}$. 

We also let $\Dcal(\R^n;\R^m) = C^\infty_c(\R^n;\R^m)$ be the space of smooth functions with compact support and $\Dcal'(\R^n;\R^m)$ be the space of distributions. We denote by $\langle \cdot, \cdot \rangle$ the duality pairing between~$\Dcal'(\R^n;\R^m)$ and~$\Dcal(\R^n;\R^m)$.

Moreover, we let~$\Schwartz(\R^d;\R^d)$ the Schwartz space of rapidly decreasing vector fields and by~$\Schwartz'(\R^d;\R^d)$ its topological dual given by the space of tempered distributions.

\subsection{Fourier transform}  \label{ihfgbvk9ikF}

The componentwise Fourier transform of a rapidly decreasing vector field~$\varphi \in \Schwartz(\R^d;\R^d)$ is defined by
\[ 
  \F[\varphi](\xi) = \widehat{\varphi}(\xi) = \frac{1}{(2\pi)^{d/2}}\int_{\R^d} \varphi(x) e^{-i x \cdot \xi} \, \d x \, , \quad \text{for every } \xi \in \R^d \, .
\]
The inverse Fourier transform of~$\psi \in \Schwartz(\R^d;\C^d)$ is given by
\[
  \F^{-1}[\psi](x) = \frac{1}{(2\pi)^{d/2}}\int_{\R^d} \psi(\eta) e^{i x \cdot \eta} \, \d \eta = \F[\psi](\xi) \Big|_{\xi = -x} \, , \quad \text{for every } x \in \R^d \, .
\]

The definition of Fourier transform is extended to tempered distributions by duality, \ie, for every~$u \in \Schwartz'(\R^d;\R^d)$ we define the Fourier transform~$\widehat{u} \in \Schwartz'(\R^d;\C^d)$ by\footnote{The duality pairing $\langle \varphi, \psi \rangle$ is intended as an extension of $\int_{\R^d} \sum_{i=1}^d \varphi_i \psi_i \, \d x$ for~$\varphi, \psi \in \Schwartz(\R^d;\C^d)$.} 
\[ 
  \langle \widehat{u}, \psi \rangle = \langle u, \widehat{\psi} \rangle \, , \quad \text{for every } \psi \in \Schwartz(\R^d;\C^d) \, .
\]
A tempered distribution~$u \in \Schwartz'(\R^d;\C^d)$ is real-valued, \ie, $u \in \Schwartz'(\R^d;\R^d)$, if taking the complex conjugation gives 
\[
\langle u, \psi \rangle = \overline{\langle u, \overline \psi \rangle} \, , \quad \text{for every } \psi \in \Schwartz(\R^d;\C^d) \, .
\]
% (This condition generalizes the condition for $u \in \Schwartz(\R^d;\C^d)$ and $\psi \in \Schwartz(\R^d;\C^d)$:
% \[
%   \begin{split}
%     \langle u, \psi \rangle & = \overline{\overline{\langle u, \psi \rangle}} = \overline{   \int_{\R^d} \sum_{i=1}^d u_i(x) \overline{\psi_i(x)} \, \d x  } = \overline{\langle u, \overline \psi \rangle} \, . )
%   \end{split}
% \]
We recall that~$u$ is real-valued if and only if~$\widehat{u} \in \Schwartz'(\R^d;\C^d)$ is Hermitian, in the sense that 
\[ \label{eq:Hermitian_Fourier_transform}
  \langle \widehat{u}, \psi(-\cdot) \rangle = \overline{ \langle \hat{u},  \overline{\psi} \rangle } \, , \quad \text{for every } \psi \in \Schwartz(\R^d;\C^d) \, .
\]
% Indeed,
% \[
%   \begin{split}
%     \langle \widehat{u}, \psi(-\cdot) \rangle & = \langle u, \widehat{\psi(-\cdot)} \rangle =  \langle u, \overline{\hat{\overline{\psi}}} \rangle = \overline{ \langle u,  \hat{\overline{\psi}} \rangle } = \overline{ \langle \hat{u},  \overline{\psi} \rangle } \, .
%   \end{split}
% \]

For functions~$\varphi(t,x)$ we let~$\widehat{\varphi}(t,\xi) = \F[\varphi(t,\cdot)](\xi)$ denote the Fourier transform with respect to the spatial variable~$x$.  

\subsection{Fractional Sobolev spaces} \label{0djofsvnk}
A classical reference for fractional Sobolev spaces is~\cite{AdaFou03}. 
In this paper we will work with the approach followed in~\cite{DNPalVal12}. We recall here the definition of the space~$H^s$ for any~$s \in \R$, introducing the notation that will be used in the rest of the paper.
\medskip

\underline{$H^s$ for $s \in (0,1)$}. Given~$s \in (0,1)$, we let $H^{s}(\R^d;\R^d)$ be the fractional Sobolev space of order~$s$ of vector fields in~$\R^d$. 
We recall that 
\[
H^s(\R^d;\R^d) = \Big\{ u \in L^2(\R^d;\R^d) \, :  \, \frac{|u(x) - u(y)|}{|x-y|^{\frac{d}{2} + s}} \in L^2(\R^d \x \R^d) \Big\} \, .  
\]
It is endowed with the norm defined by 
\[ 
  \|u\|_{H^s} = \left( \|u\|_{L^2}^2 + [u]_{H^s}^2 \right)^\frac{1}{2} \, ,
\]
where 
\[ 
[u]_{H^s} = \left( \int_{\R^d}\int_{\R^d} \frac{|u(x) - u(y)|^2}{|x-y|^{d + 2s}} \, \d x \, \d y \right)^\frac{1}{2}
\]
is the so-called Gagliardo seminorm.
The norm on $H^s(\R^d;\R^d)$ is induced by an inner product that makes~$H^s(\R^d;\R^d)$ a Hilbert space.
\medskip

\underline{$H^s$ for $s > 1$}. If~$s > 1$ and~$s$ is not an integer, then~$s = m + \sigma$, where~$m \in \N$ and~$\sigma \in (0,1)$. 
Then the space~$H^s(\R^d;\R^d)$ is the space of vector fields~$u \in H^m(\R^d;\R^d)$ such that $\D^\mathbf{k} u \in H^\sigma(\R^d;\R^d)$ for every multi-index~$\mathbf{k}$ with~$|\mathbf{k}| = m$.

This space is endowed with the norm defined by 
\[ 
  \|u\|_{H^s} = \left(\|u\|_{H^{m}}^2 + \sum_{|\mathbf{k}| = m}  [\D^\mathbf{k} u]_{H^\sigma}^2 \right)^\frac{1}{2}\, ,
\]
which is induced by an inner product that makes $H^s(\R^d;\R^d)$ a Hilbert space.
\medskip

\underline{$H^s$ for $s < 0$}. If~$s < 0$, then~$H^s(\R^d;\R^d)$ is the space of distributions given by the dual of~$H^{-s}(\R^d;\R^d)$.  
\medskip

\underline{$H^s$ and Fourier transform}. 
For every~$s \in \R$, the Hilbert space~$H^s(\R^d;\R^d)$ is equivalently characterized as the subspace of tempered distributions 
\[  \label{eq:Hs_in_Fourier_space}
  H^s(\R^d;\R^d) = \big\{ u \in \Schwartz'(\R^d;\R^d) \, : \, (1 + |\xi|^2)^{\frac{s}{2}} |\widehat{u}(\xi)| \in L^2(\R^d) \big\} \, .
\]
Moreover, we recall that
\[ \label{eq:Hs_norm_equivalence}
  \frac{1}{C}\|(1 + |\xi|^2)^{\frac{s}{2}} \widehat{u}\|_{L^2} \leq \|u\|_{H^s} \leq C \|(1 + |\xi|^2)^{\frac{s}{2}} \widehat{u}\|_{L^2} \, ,
\]
see, \eg,~\cite[Proposition~3.4]{DNPalVal12} for the case~$s \in (0,1)$.

\section{The peridynamics operator} \label{sec:peridynamics_operator}

In this section we define the main ingredient of the peridynamics model, \ie, the peridynamics operator, and we collect some of its properties used throughout the paper.

\subsection{Peridynamics horizon} 
To model the size of the interaction region for bonds, we introduce a cut-off function~$\chi \in L^\infty(\R)$ with 
\[
\supp(\chi) \subset [-1,1] \, , \quad  0 \leq \chi \leq 1 \, , \quad \chi \equiv 1 \text{ on } \left[-\frac{1}{2},\frac{1}{2}\right] \, .
\] 
The assumption~$\chi \equiv 1$ in~$\big[-\frac{1}{2},\frac{1}{2}\big]$ might be relaxed by requiring that~$\chi \equiv 1$ on a neighborhood of the origin or that~$\chi(0) = 1$ and~$\chi$ is of class~$C^3$ in a neighborhood of~$0$. We prefer to keep the more restrictive assumption to make some computations easier to follow.

We let~$\delta > 0$ be a parameter modelling the size of the interaction region, named the \emph{peridynamics horizon}. This parameter measures the distance of the nonlocal bonds in the peridynamics model.
 In this paper we are interested to the regime where~$\delta$ is a small parameter. 

For every~$\delta > 0$ we define the function 
\[ \label{eq:def_chi_delta}
\chi_\delta(y) := \chi\left(\frac{|y|}{\delta}\right) \, , \quad \text{for every } y \in \R^d \, .
\] 
Throughout the paper, we shall often use the following properties of~$\chi_\delta$: 
\[ \label{eq:properties_chi_delta}
  \supp(\chi_\delta) \subset B_\delta  \, , \quad 0 \leq \chi_\delta \leq 1 \, , \quad  \chi_\delta \equiv 1 \text{ on } B_{\delta/2} \, .
\]

\begin{remark}
  In this paper we are considering the case where~$\chi$ is a radially symmetric function. 
  Under this assumption, some computations are simplified and we can focus on the main ideas in the proofs. 
  Moreover, it is due to this assumption that we obtain the Laplace operator in the limit as~$\delta \to 0$, see also Lemma~\ref{lem:limit_K_delta} below. 

  It should be noted that the results in this paper can be extended to the case where~$\chi$ is not radially symmetric, but only symmetric with respect to the origin.
  In that case the limit operator might be an anisotropic version of the Laplacian.  
\end{remark}
 
\subsection{Peridynamics operator} \label{SEC:OPE}
The main feature of the peridynamics model is the presence of a nonlocal operator~$K_\delta$ that will be used to model the internal force field in the material as a result of the displacement field~$u \colon \R^d \to \R^d$.
Some technical comments on the nonlocal operator~$K_\delta[u]$ introduced in~\eqref{eq:def_K_delta} are in order.

\begin{remark}
  Since $\supp(\chi_\delta) \subset B_\delta$, the integral in~\eqref{eq:def_K_delta} is actually an integral on~$B_\delta$, \ie, until the peridynamics horizon.
\end{remark}

\begin{remark}
  The operator~$K_\delta$ resembles a fractional Laplacian.
   However, the presence of the peridynamics horizon~$\delta$ makes it different from the classical fractional Laplacian. 
  The dependence of the nonlocal operator on~$\delta$ is crucial, as in this paper we are interested precisely in the behavior of the model depending on~$\delta$.
\end{remark}

\begin{remark}
Due to the singularity of the integrand, the integral in~\eqref{eq:def_K_delta} is meant in the principal value sense. 
  More precisely, 
  \[
    K_\delta[u](x) = - \lim_{\e \to 0} \frac{\kappa}{\delta^{2(1-\alpha)}} \int_{\R^d \setminus B_\e(0)} \chi_\delta(y) \frac{u(x) - u(x-y)}{|y|^{d + 2 \alpha}} \, \d y \, .
  \]
  In the particular case where~$\alpha \in (0,\frac{1}{2})$, then the principal value integral coincides with the standard integral for~$u \in \Schwartz(\R^d;\R^d)$, since the singularity cancels out. 
  Indeed, since $\supp(\chi_\delta) \subset B_\delta$, integrating in spherical coordinates we obtain that 
  \[ 
    \begin{split}
      & \int_{\R^d \setminus B_\e(0)} \chi_\delta(y) \frac{|u(x) - u(x-y)|}{|y|^{d + 2 \alpha}} \, \d y \leq \|\D u\|_{L^\infty} \int_{B_\delta(0) \sm B_\e(0)} \frac{1}{|y|^{d + 2 \alpha - 1}} \, \d y  \\
      & \quad  =  \|\D u\|_{L^\infty}  \int_{\e}^{\delta} \frac{1}{\rho^{2 \alpha}} \, \, \d \rho \leq  \frac{\|\D u\|_{L^\infty}}{1-2\alpha} \delta^{1-2 \alpha} \, ,
    \end{split}
    \]
    for every~$\e > 0$.
  \end{remark}

  An alternative expression for the singular integral in~\eqref{eq:def_K_delta} as a weighted second order difference quotient allows us to remove the singularity in the integrand at the origin and avoid using the principal value, when needed. The precise statement is given in the following lemma. 

  \begin{lemma} \label{lem:K_delta_alternative}
    Let~$u \in \Schwartz(\R^d;\R^d)$. Then,
    \[ \label{eq:K_delta_alternative}
      K_\delta[u](x) = \frac{1}{2} \frac{\kappa}{\delta^{2(1-\alpha)}} \int_{\R^d} \chi_\delta(y) \frac{u(x+y) + u(x-y) - 2 u(x)}{|y|^{d + 2 \alpha}} \, \d y \, , \quad \text{for every } x \in \R^d \, ,
    \]
    where the integral is meant in the classical sense and is finite. Moreover, $K_\delta[u] \in L^1(\R^d;\R^d)$.
  \end{lemma}
  
  \begin{proof}  
    Let~$u \in \Schwartz(\R^d;\R^d)$. By the symmetry of~$\chi_\delta$ we have that~$\chi_\delta(y) = \chi_\delta(-y)$ for every~$y \in \R^d$. Given~$u \in \Schwartz(\R^d;\R^d)$, a change of variables then yields that
    \[ \label{eq:K_delta_with_plus}
      \begin{split}
        K_\delta[u](x) & = - \frac{\kappa}{\delta^{2(1-\alpha)}} \pvint_{\R^d} \chi_\delta(y) \frac{u(x) - u(x-y)}{|y|^{d + 2 \alpha}} \, \d y \\
        & = - \frac{\kappa}{\delta^{2(1-\alpha)}} \pvint_{\R^d} \chi_\delta(y) \frac{u(x) - u(x+y)}{|y|^{d + 2 \alpha}} \, \d y \, .
      \end{split}
    \]
    It follows that
    \[ \label{eq:K_delta_alternative_pv}
      \begin{split}
        &K_\delta[u](x) \\
         = \;&- \frac{1}{2} \frac{\kappa}{\delta^{2(1-\alpha)}} \left( \pvint_{\R^d} \chi_\delta(y) \frac{u(x) - u(x-y)}{|y|^{d + 2 \alpha}} \, \d y +  \pvint_{\R^d} \chi_\delta(y) \frac{u(x) - u(x+y)}{|y|^{d + 2 \alpha}} \, \d y \right) \\
 =\;& \frac{1}{2} \frac{\kappa}{\delta^{2(1-\alpha)}} \pvint_{\R^d} \chi_\delta(y) \frac{u(x+y) + u(x-y) - 2 u(x)}{|y|^{d + 2 \alpha}} \, \d y \, .\\
      \end{split}
    \]

    Let us show that the principal value integral in the last expression is actually an integral in the classical sense. Indeed, by a Taylor expansion of~$u$ we have that, for~$y \in B_\delta$,
      \[
      \begin{split}
        u(x+y) & = u(x) + \D u(x)  y + O(|y^2|) \\
  {\mbox{and }}\quad      u(x-y) & = u(x) - \D u(x)   y + O(|y^2|) \, ,
      \end{split}
      \]
      where, more precisely, the error terms are of order~$C|y|^2\sup_{z \in B_\delta(x)} |\D^2 u|$. 
      Since~$u \in \Schwartz(\R^d;\R^d)$, we have that 
      \[
        \sup_{z \in B_\delta(x)} |\D^2 u| \leq \frac{C_\delta}{1 + |x|^{d+1}} \, , \quad \text{for every } x \in \R^d \, ,
      \]
      for some constant~$C_\delta > 0$ depending on~$\delta$.
      Hence, for every~$x \in \R^d$ and~$y \in B_\delta$, we have that
      \[ \label{eq:u_expansion_in_L1}
        \frac{|u(x+y) + u(x-y) - 2 u(x)|}{|y|^{d + 2 \alpha}}  \leq \frac{C\sup_{z \in B_\delta(x)} |\D^2 u|}{|y|^{d+ 2 \alpha -2}} \leq C_\delta \frac{1}{1+|x|^{d+1}} \frac{1}{|y|^{d+ 2 \alpha -2}} \, .
      \]

      The above estimate shows two things. 
      First, the term~$\frac{1}{1+|x|^{d+1}}$ is bounded and the function $\frac{1}{|y|^{d+ 2 \alpha -2}}$ is integrable with respect to~$y$ in~$B_\delta$, since~$d + 2 \alpha - 2 < d$.
      We deduce that the principal value integral in the expression of~$K_\delta[u]$ in~\eqref{eq:K_delta_alternative_pv} is actually an integral in the classical sense, \ie, 
      \[  
        K_\delta[u](x) = \frac{1}{2} \frac{\kappa}{\delta^{2(1-\alpha)}} \int_{\R^d} \chi_\delta(y) \frac{u(x+y) + u(x-y) - 2 u(x)}{|y|^{d + 2 \alpha}} \, \d y \, ,
      \] 
      and it is finite. 

      Second, \eqref{eq:u_expansion_in_L1} shows that 
      \[  \label{eq:integrand_K_delta_in_L1}
        \chi_\delta(y) \frac{u(x+y) + u(x-y) - 2 u(x)}{|y|^{d + 2 \alpha}} \in L^1(\R^d \x \R^d; \R^d) \, .
      \]
      By Fubini's Theorem, we conclude that~$K_\delta[u] \in L^1(\R^d;\R^d)$.
  \end{proof}
    
  \begin{remark}
   Since the nonlocal operator~$K_\delta[u]$ is the main ingredient in the peridynamics model, it is convenient to recast it in a form that highlights the constitutive assumptions of the model.
    More precisely, it can be written as a convolution operator with a singular kernel~$f_\delta$, \ie, 
    \[ \label{eq:K_delta_convolution}
      K_\delta[u](x) = - \frac{\kappa}{\delta^{2(1-\alpha)}} \pvint_{\R^d} \chi_\delta(x-x') \frac{u(x) - u(x')}{|x-x'|^{d + 2 \alpha}} \, \d x' = - \pvint_{\R^d} f_\delta(x-x',u(x) - u(x')) \, \d x' \, ,
    \]
    where $f_\delta \colon (\R^d \sm \{0\}) \times \R^d \to \R^d$ is defined by
    \[
      f_\delta(z,u) := \frac{\kappa}{\delta^{2(1-\alpha)}} \chi_\delta(z) \frac{u}{|z|^{d + 2 \alpha}} \, , \quad \text{for every } (z,u) \in (\R^d \sm \{0\}) \times \R^d \, .
    \]
    We observe that the singular kernel~$f_\delta$ satisfies the following properties, which are the typical constitutive assumptions in peridynamics models:
    \begin{itemize}
      \item  $f_\delta(-z,-u) = -f_\delta(z,u)$ for every~$(z,u) \in (\R^d \sm \{0\}) \times \R^d$ (the counterpart of Newton's law of action and reaction);
      \item $f_\delta = \nabla_u \Phi_\delta$ (the counterpart of the assumption that forces are conservative), where the potential function $\Phi_\delta \colon (\R^d \sm \{0\}) \times \R^d \to \R$ is defined by 
      \[
        \Phi_\delta(z,u) = \frac{1}{2} \frac{\kappa}{\delta^{2(1-\alpha)}} \chi_\delta(z)  \frac{|u|^2}{|z|^{d + 2 \alpha}} \, , \quad \text{for every } (z,u) \in (\R^d \sm \{0\}) \times \R^d \, .
      \]
    \end{itemize}
\end{remark}

\subsection{Limit as $\delta \to 0$ of the operator}
  The presence of the scaling factor~$\delta^{2(1-\alpha)}$ in the definition of~$K_\delta[u]$ is crucial for the main results in this paper to obtain a nontrivial limit as~$\delta \to 0$. To clarify the specific choice of the scaling factor we present the following result.

 \begin{lemma} \label{lem:limit_K_delta}
  Let $u \in \Schwartz(\R^d;\R^d)$. Then,
  \[
    \lim_{\delta \to 0}  K_\delta[u](x) = \gamma^2 \Delta u(x) \, , \quad \text{for every } x \in \R^d \, ,
  \]
  where 
  \[ \label{eq:def_gamma}
  \gamma = \left( \frac{\kappa}{2} \L^d(B_1) \int_{0}^1 \chi(\rho) \rho^{1-2 \alpha}  \, \d \rho \right)^\frac{1}{2} \, .
  \]
 \end{lemma}
 \begin{proof}
  We prove the result componentwise, \ie, we show that~$(K_\delta[u])_i(x) \to \gamma^2 \Delta u_i(x)$ for every~$i = 1,\ldots,d$. 
  We approximate a component~$u_i$ using a Taylor expansion, \ie, we write
  \[
    u_i(x-y) = u_i(x) - \nabla u_i(x) \cdot y + \frac{1}{2} (\D^2 u_i(x) y) \cdot y + O(|y|^3) \, .
  \]
  We then substitute this expression in the definition
  of~$K_\delta[u](x)$ in~\eqref{eq:def_K_delta}, we use the definition of~$\chi_\delta$ in~\eqref{eq:def_chi_delta}, and we apply the change of variables~$y = \delta z$ to obtain that
  \[  \label{eq:K_delta_expansion}
    \begin{split}
      (K_\delta[u])_i(x) & = - \frac{\kappa}{\delta^{2(1-\alpha)}} \pvint_{\R^d} \chi_\delta(y) \frac{\nabla u_i(x) \cdot y - \frac{1}{2} (\D^2 u_i(x) y) \cdot y + O(|y|^3)}{|y|^{d + 2 \alpha}} \, \d y  \\
      & = - \frac{\kappa}{\delta^{2(1-\alpha)}} \pvint_{\R^d} \chi_\delta(\delta z) \frac{\nabla u_i(x) \cdot z - \delta^2 \frac{1}{2} (\D^2 u_i(x)  z) \cdot   z + \delta^3 O(|z|^3)}{\delta^{d+2\alpha} | z|^{d + 2 \alpha}} \delta^d \, \d z  \\
      & = - \frac{\kappa}{\delta} \pvint_{\R^d} \chi(|z|) \frac{\nabla u_i(x) \cdot z - \delta \frac{1}{2} (\D^2 u_i(x)  z) \cdot   z + \delta^2 O(|z|^3)}{| z|^{d + 2 \alpha}} \, \d z \, .
    \end{split}
  \]
  
  We treat the three terms separately. 
  For the first term, we have that
  \[
    \pvint_{\R^d} \chi(|z|) \frac{\nabla u(x) \cdot z}{|z|^{d + 2 \alpha}} \, \d z  = 0\, .
  \]
  To show this, let~$R \in SO(d)$ be a rotation such that $R^T\nabla u_i(x) = |\nabla u_i(x)| e_1$, where~$e_1 = (1,0,\ldots,0) \in \R^d$. Then, we apply the change of variables~$w = Rz$ to infer that
  \[
    \pvint_{\R^d} \chi(|z|) \frac{\nabla u_i(x) \cdot z}{|z|^{d + 2 \alpha}} \, \d z = |\nabla u_i(x)|  \pvint_{\R^d} \chi(|z|)  \frac{z_1}{|z|^{d + 2 \alpha}} \, \d z  = 0  
  \]
  by symmetry. 

  For the second term in~\eqref{eq:K_delta_expansion}, we apply the spectral theorem to~$\D^2 u_i(x)$ to find a matrix~$Q \in O(d)$ such that~$Q^T \D^2 u_i(x) Q = \Lambda$, where~$\Lambda$ is the diagonal matrix with the eigenvalues~$\lambda_1,\ldots,\lambda_d$ of~$\D^2 u_i(x)$ on the diagonal. We then apply the change of variables~$z = Qw$ to obtain that
  \[
    \begin{split}
      & \pvint_{\R^d} \chi(|z|) \frac{(\D^2 u_i(x) z) \cdot z}{|z|^{d + 2 \alpha}} \, \d z = \pvint_{\R^d} \chi(|w|) \frac{(\D^2 u_i(x) Qw) \cdot Qw}{|w|^{d + 2 \alpha}} \, \d w  \\
      & \quad =  \pvint_{\R^d} \chi(|w|) \frac{(\Lambda w) \cdot w}{|w|^{d + 2 \alpha}} \, \d w = \sum_{j=1}^d \lambda_i \pvint_{\R^d} \chi(|w|) \frac{w_j^2}{|w|^{d + 2 \alpha}} \, \d w  \, .
    \end{split}
  \]
  By symmetry and by integrating in spherical coordinates, for~$j=1,\dots,d$ we have that
  \[
    \begin{split}
      & \kappa \pvint_{\R^d} \chi(|w|) \frac{w_j^2}{|w|^{d + 2 \alpha}} \, \d w = \frac{\kappa}{d} \sum_{\ell=1}^d \pvint_{\R^d} \chi(|w|) \frac{w_\ell^2}{|w|^{d + 2 \alpha}} \, \d w = \frac{\kappa}{d} \pvint_{\R^d} \chi(|w|) \frac{|w|^2}{|w|^{d + 2 \alpha}} \, \d w \\
      & \quad = \frac{\kappa}{d} \pvint_{\R^d}  \frac{\chi(|w|)}{|w|^{d + 2 \alpha - 2}} \, \d w = \kappa \frac{\H^{d-1}(\S^{d-1})}{d}  \pvint_{0}^\infty \frac{\chi(\rho)}{\rho^{d + 2 \alpha - 2}} \rho^{d-1}\, \d \rho\\
      & \quad  = \kappa \L^d(B_1) \pvint_{0}^\infty  \chi(\rho) \rho^{1-2 \alpha}  \, \d \rho = \kappa \L^d(B_1) \int_{0}^1  \chi(\rho) \rho^{1-2 \alpha}  \, \d \rho = 2\gamma^2 \, ,
    \end{split}
  \]
  since~$\alpha \in (0,1)$ and~$\chi$ has compact support. We have shown that 
  \[
    \begin{split}
      & \frac{\kappa}{\delta}\pvint_{\R^d} \chi(|z|) \frac{\delta \frac{1}{2} (\D^2 u_i(x) z) \cdot z}{|z|^{d + 2 \alpha}} \, \d z = \frac{1}{2} \kappa \int_{\R^d} \chi(|z|) \frac{(\D^2 u_i(x) z) \cdot z}{|z|^{d + 2 \alpha}} \, \d z =  \gamma^2 \sum_{j=1}^d \lambda_j \\
      & = \gamma^2 \, \mathrm{trace}(\D^2 u(x)) = \gamma^2 \Delta u(x) \, .
    \end{split}
  \]

  Finally, let us show that the third term in~\eqref{eq:K_delta_expansion} vanishes as~$\delta \to 0$. Indeed, the function~$\frac{|z|^3}{|z|^{d + 2\alpha}} = \frac{1}{|z|^{d  + 2\alpha - 3}}$ is integrable in~$B_1$ since~$d + 2\alpha - 3 < d$. It follows that 
    \[
    \lim_{\delta\to0}  \frac{\kappa}{\delta} \pvint_{\R^d} \chi(|z|) \frac{ \delta^2 O(|z|^3)}{| z|^{d + 2 \alpha}} \, \d z \leq  \lim_{\delta\to0}
    \kappa \delta C \int_{\R^d}  \frac{\chi(|z|)}{| z|^{d + 2 \alpha-3}} \, \d z = 0 \, .  
    \]

Thus, gathering these pieces of information and letting~$\delta \to 0$ in~\eqref{eq:K_delta_expansion} we obtain the desired result.
 \end{proof}

 \begin{remark}
  The constant~$\gamma$ in Lemma~\ref{lem:limit_K_delta} will represent the speed of propagation of waves in the limit~$\delta \to 0$. It depends on the dimension~$d$, on the constitutive parameters~$\kappa$ and~$\alpha$, and on the cut-off function~$\chi$.
   In the particular case where~$\chi$ is the characteristic function of the interval~$[-1,1]$, we have the explicit formula~$\gamma = \big(\frac{\kappa}{4(1-\alpha)}   \L^d(B_1) \big)^\frac{1}{2}$.
 \end{remark}

\subsection{Dual operator and symmetry} 
In the sequel we will make use the dual operator of~$K_\delta$. 
It is needed to define the notion of distributional solution to the peridynamics model in Definition~\ref{def:distributional_solution} and to prove the uniqueness result in Theorem~\ref{thm:uniqueness} below.

The dual operator~$K_\delta^*$ of~$K_\delta$ is defined by
\[ \label{eq:def_K_delta_star}
K_\delta^*[v](x) := - \frac{\kappa}{\delta^{2(1-\alpha)}} \pvint_{\R^d} \chi_\delta(y) \frac{v(x) - v(x+y)}{|y|^{d + 2 \alpha}} \, \d y \, , \quad \text{for every } x \in \R^d \, ,
\] 
for every~$v \in \Schwartz(\R^d;\R^d)$.

\begin{remark}
  By the symmetry of~$\chi_\delta$, we have that~$K_\delta^*[v] = K_\delta[v]$ for every~$v \in \Schwartz(\R^d;\R^d)$. (See formula~\eqref{eq:K_delta_with_plus} for the computation.) This means that the operator~$K_\delta$ is symmetric.
\end{remark}

\begin{remark} \label{rem:symmetry}
  The dual operator~$K_\delta^*$ is defined in such a way that the duality pairing in the sense of distributions satisfies
    \[ \label{eq:duality_K_K_star}
      \langle K_\delta[u],v \rangle = \langle u,K_\delta^*[v] \rangle \, ,
  \] 
  for every $u \in \Schwartz(\R^d;\R^d)$ and~$v \in \Schwartz(\R^d;\R^d)$. Let us show this by using the alternative expression for~$K_\delta$ deduced in Lemma~\ref{lem:K_delta_alternative}. We start by observing that 
  \[
    \chi_\delta(y)\frac{u(x+y) + u(x-y) - 2 u(x)}{|y|^{d + 2 \alpha}} \cdot  v(x) \in L^1(\R^d \x \R^d) \, . 
  \]
  Then we can compute the duality pairing. By Fubini's Theorem we have that
  \[
    \begin{split}
     \langle K_\delta[u],v \rangle & = \int_{\R^d} K_\delta[u](x) \cdot v(x) \, \d x  \\
      & = \int_{\R^d} \left( \frac{1}{2} \frac{\kappa}{\delta^{2(1-\alpha)}} \int_{\R^d} \chi_\delta(y) \frac{u(x+y) + u(x-y) - 2 u(x)}{|y|^{d + 2 \alpha}} \, \d y \right) \cdot v(x) \, \d x \, . \\
      & = \frac{1}{2} \frac{\kappa}{\delta^{2(1-\alpha)}}  \int_{\R^d} \frac{\chi_\delta(y)}{|y|^{d + 2 \alpha}} \left( \int_{\R^d}   \big( u(x+y) + u(x-y) - 2 u(x) \big)  \cdot v(x) \, \d x \right) \, \d y \\
      & = \frac{1}{2} \frac{\kappa}{\delta^{2(1-\alpha)}}  \int_{\R^d} \frac{\chi_\delta(y)}{|y|^{d + 2 \alpha}} \left( \int_{\R^d}   \big( v(x+y) + v(x-y) - 2 v(x) \big)  \cdot u(x) \, \d x \right) \, \d y \\
      & =\int_{\R^d} \left( \frac{1}{2} \frac{\kappa}{\delta^{2(1-\alpha)}} \int_{\R^d} \chi_\delta(y) \frac{v(x+y) + v(x-y) - 2 v(x)}{|y|^{d + 2 \alpha}} \, \d y \right) \cdot u(x) \, \d x \\
      & = \int_{\R^d} u(x) \cdot K_\delta[v](x)  \, \d x  \\&= \langle u, K_\delta[v] \rangle \, ,
    \end{split}
  \]
  where we applied a change of variables in the fourth equality. Since~$K_\delta = K_\delta^*$, this concludes the proof.  
\end{remark}

\subsection{Symbol of the peridynamics operator}
We compute here the symbol of the nonlocal operator~$K_\delta$ in the Fourier domain. 
This will be useful to provide solutions to the peridynamics model through the Fourier transform. 

\begin{lemma} \label{lem:symbol_K_delta}
  Let $u \in \Schwartz(\R^d;\R^d)$. Then,
\[
  \hat{K_\delta[u]}(\xi) = - \omega_\delta^2(\xi) \hat{u}(\xi) \, , 
\]
where 
\[ \label{eq:def_omega_delta}
\omega_\delta(\xi) = \left( \frac{\kappa}{\delta^{2(1-\alpha)}} \int_{\R^d} \chi_\delta(y) \frac{1 - \cos(y \cdot \xi)}{|y|^{d + 2 \alpha}} \, \d y \right)^\frac{1}{2} \, .
\]
\end{lemma}

\begin{proof}
  First of all, we observe that Lemma~\ref{lem:K_delta_alternative} allows us to compute the Fourier transform of~$K_\delta[\varphi]$ since~$K_\delta[u] \in L^1(\R^d;\R^d)$. Using the alternative expression of~$K_\delta[u]$ in~\eqref{eq:K_delta_alternative} and by~\eqref{eq:integrand_K_delta_in_L1}, we apply Fubini's Theorem to obtain that
  \[
    \begin{split}
      \hat{K_\delta[u]}(\xi) & = \frac{1}{(2\pi)^{d/2}}\int_{\R^d} K_\delta[u](x) e^{-i x \cdot \xi}  \, \d x \\
      & = \frac{1}{(2\pi)^{d/2}}\int_{\R^d} \left( \frac{1}{2} \frac{\kappa}{\delta^{2(1-\alpha)}} \int_{\R^d} \chi_\delta(y) \frac{u(x+y) + u(x-y) - 2 u(x)}{|y|^{d + 2 \alpha}} \, \d y  \right) e^{-i x \cdot \xi}  \, \d x \\
      & = \frac{1}{2} \frac{\kappa}{\delta^{2(1-\alpha)}}  \int_{\R^d} \frac{\chi_\delta(y)}{|y|^{d + 2 \alpha}}  \left( \frac{1}{(2\pi)^{d/2}}  \int_{\R^d} \big( u(x+y) + u(x-y) - 2 u(x) \big) e^{-i x \cdot \xi}  \, \d x  \right)   \, \d y \\
      & = \left( \frac{1}{2} \frac{\kappa}{\delta^{2(1-\alpha)}}  \int_{\R^d} \chi_\delta(y) \frac{e^{i y \cdot \xi} + e^{- i y \cdot \xi} - 2}{|y|^{d + 2 \alpha}}      \, \d y \right) \hat{u}(\xi)  \\
      & = - \left( \frac{\kappa}{\delta^{2(1-\alpha)}}  \int_{\R^d} \chi_\delta(y) \frac{1 - \cos(y \cdot \xi)}{|y|^{d + 2 \alpha}} \, \d y \right) \hat{u}(\xi) \, .
    \end{split}
  \]
    This concludes the proof.
\end{proof}

We point out that the function~$\omega_\delta(\xi)$ obtained in Lemma~\ref{lem:symbol_K_delta} is the dispersion relation in the peridynamics model. 
Because of its importance, we will study here some of its properties.

\begin{remark} \label{rem:omega_delta_radial}
  Since~$\chi_\delta$ is radial, then~$\omega_\delta$ is radial as well. Indeed, for any~$R \in O(d)$, by changing variables~$z = R^T y$ we have that
  \[
    \begin{split}
      \omega_\delta^2(R\xi) &  = \frac{\kappa}{\delta^{2(1-\alpha)}} \int_{\R^d} \chi_\delta(y) \frac{1 - \cos(y \cdot R\xi)}{|y|^{d + 2 \alpha}} \, \d y = \frac{\kappa}{\delta^{2(1-\alpha)}} \int_{\R^d} \chi_\delta(y) \frac{1 - \cos(R^Ty \cdot \xi)}{|y|^{d + 2 \alpha}} \, \d y \\
      & = \frac{\kappa}{\delta^{2(1-\alpha)}} \int_{\R^d} \chi_\delta(Rz) \frac{1 - \cos(z \cdot \xi)}{|Rz|^{d + 2 \alpha}} \, \d z = \frac{\kappa}{\delta^{2(1-\alpha)}} \int_{\R^d} \chi_\delta(z) \frac{1 - \cos(z \cdot \xi)}{|z|^{d + 2 \alpha}} \, \d z = \omega_\delta^2(\xi) \, .
    \end{split}
    \]
\end{remark}
  
A first important result concerning the dispersion relation~$\omega_\delta$ is its behavior for small and large frequencies~$\xi$. We see that the nonlocality strongly affects the behavior of~$\omega_\delta$. This analysis has been carried out in~\cite{CocDipFanMadVal22}. We provide the details of the proof for the sake of completeness.
\begin{lemma} \label{lem:omega_delta_low_high_frequencies}
  We have that 
  \[ \label{eq:omega_delta_low_frequencies}
  \lim_{|\xi| \to 0} \frac{\omega_\delta(\xi)}{|\xi|} = \gamma = \left( \frac{\kappa}{2} \L^d(B_1) \int_{0}^1 \chi(\rho) \rho^{1-2 \alpha}  \, \d \rho \right)^\frac{1}{2}\, ,
  \]
  and 
  \[ \label{eq:omega_delta_high_frequencies}
  \lim_{|\xi| \to +\infty} \frac{\omega_\delta(\xi)}{|\xi|^\alpha} = \lambda = \left( \frac{\kappa}{\delta^{2(1-\alpha)}} \int_{\R^d}  \frac{1 - \cos(z_1)}{|z|^{d + 2 \alpha}}   \, \d z \right)^\frac{1}{2} \, .  
  \]
\end{lemma}
\begin{proof}
  In this proof, we write~$\xi = |\xi| \nu$ with~$\nu \in \S^{d-1}$. 
  
  Let us start by proving~\eqref{eq:omega_delta_low_frequencies}. We have that
  \[
    \begin{split}
       \frac{\omega_\delta^2(\xi)}{|\xi|^2} & =  \frac{1}{|\xi|^2}\frac{\kappa}{\delta^{2(1-\alpha)}} \int_{\R^d} \chi_\delta(y) \frac{1 - \cos(|\xi|y \cdot \nu)}{|y|^{d + 2 \alpha}} \, \d y \, .
    \end{split}
  \]
  We observe that 
  \[
    \chi_\delta(y)\frac{|1 - \cos(|\xi|y \cdot \nu)|}{|y|^{d + 2 \alpha}} \leq  |\xi|^2 \frac{ \chi_\delta(y)}{2|y|^{d + 2 \alpha - 2}}  \, , 
  \]
  and the function in the right-hand side is integrable in~$B_\delta$. We can then apply the Dominated Convergence Theorem to obtain that
  \[
    \lim_{|\xi| \to 0} \frac{\omega_\delta^2(\xi)}{|\xi|^2} = \frac{\kappa}{\delta^{2(1-\alpha)}} \int_{\R^d} \frac{\chi_\delta(y)}{|y|^{d + 2 \alpha}} \left(
    \lim_{|\xi| \to 0}  \frac{1 - \cos(|\xi|y \cdot \nu)}{|\xi|^2} \right) \, \d y = \frac{\kappa}{2 \delta^{2(1-\alpha)}} \int_{\R^d} \chi_\delta(y) \frac{|y \cdot \nu|^2}{|y|^{d + 2 \alpha}}  \, \d y  \, .
  \]
  Let us compute the integral obtained in the limit. By radial symmetry, by substituting~$y = \delta z$, and by integrating in spherical coordinates, we have that
  \[ \label{eq:gamma_alternative_formula}
    \begin{split}
      & \frac{\kappa}{2 \delta^{2(1-\alpha)}}  \int_{\R^d} \chi_\delta(y) \frac{|y \cdot \nu|^2}{|y|^{d + 2 \alpha}}  \, \d y  = \frac{\kappa}{2 \delta^{2(1-\alpha)}}  \int_{\R^d} \chi_\delta(y) \frac{y_1^2}{|y|^{d + 2 \alpha}}  \, \d y \\
      & \quad  = \frac{1}{d} \frac{\kappa}{2 \delta^{2(1-\alpha)}}  \sum_{i=1}^d \int_{\R^d} \chi_\delta(y) \frac{y_i^2}{|y|^{d + 2 \alpha}}  \, \d y  = \frac{1}{d} \frac{\kappa}{2 \delta^{2(1-\alpha)}}  \int_{\R^d} \chi
      \left(\frac{|y|}{\delta}\right) \frac{1}{|y|^{d + 2 \alpha-2}}  \, \d y \\
      & \quad = \frac{1}{d} \frac{\kappa}{2 \delta^{2(1-\alpha)}}  \int_{\R^d} \chi(|z|) \frac{1}{\delta^{d + 2 \alpha-2}|z|^{d + 2 \alpha-2}}  \delta^d \, \d z =  \frac{\kappa}{2}  \frac{1}{d} \int_{\R^d} \chi(|z|) \frac{1}{|z|^{d + 2 \alpha-2}}    \, \d z \\
      & \quad = \frac{\kappa}{2} \frac{\H^{d-1}(\S^{d-1})}{d}  \int_0^{+\infty} \chi(\rho) \frac{1}{\rho^{d + 2 \alpha-2}} \rho^{d-1}  \, \d \rho = \frac{\kappa}{2} \L^d(B_1) \int_0^1 \chi(\rho) \rho^{1-2 \alpha}  \, \d \rho = \gamma^2 \, .
    \end{split}
  \]
  The last two displays give~\eqref{eq:omega_delta_low_frequencies}.

  Let us prove now . Applying the change of variables~$|\xi|y = z$, we have that 
  \[ \label{eq:omega_delta_high_frequencies_computation}
    \begin{split}
      \frac{\omega_\delta^2(\xi)}{|\xi|^{2\alpha}} & = \frac{1}{|\xi|^{2\alpha}}\frac{\kappa}{\delta^{2(1-\alpha)}} \int_{\R^d} \chi_\delta(y) \frac{1 - \cos(|\xi|y \cdot \nu)}{|y|^{d + 2 \alpha}} \, \d y \\
      & =  \frac{\kappa}{\delta^{2(1-\alpha)}} \int_{\R^d} \chi_\delta\left(\frac{z}{|\xi|}\right) \frac{1 - \cos(z \cdot \nu)}{|z|^{d + 2 \alpha}}   \, \d z \, .
    \end{split}
  \]
  Taking the limit as~$|\xi| \to \infty$, by the Dominated Convergence Theorem we obtain that
  \[
    \begin{split}
      \lim_{|\xi| \to +\infty} \frac{\omega_\delta^2(\xi)}{|\xi|^{2\alpha}} & =  \frac{\kappa}{\delta^{2(1-\alpha)}} \chi_\delta(0) \int_{\R^d}  \frac{1 - \cos(z \cdot \nu)}{|z|^{d + 2 \alpha}}   \, \d z = \frac{\kappa}{\delta^{2(1-\alpha)}} \int_{\R^d}  \frac{1 - \cos(z \cdot \nu)}{|z|^{d + 2 \alpha}}   \, \d z  \\
      & \qquad= \frac{\kappa}{\delta^{2(1-\alpha)}} \int_{\R^d}  \frac{1 - \cos(z_1)}{|z|^{d + 2 \alpha}}   \, \d z = \lambda^2 \, ,
    \end{split}
  \]
  where we used that~$\chi_\delta(0) = 1$ by~\eqref{eq:properties_chi_delta}.
  This establishes~\eqref{eq:omega_delta_high_frequencies}, as desired.
\end{proof}

\begin{remark} \label{rem:regularity_of_K_delta}
  A consequence of Lemma~\ref{lem:symbol_K_delta} and Lemma~\ref{lem:omega_delta_low_high_frequencies} is that
  \[
    u \in \Schwartz(\R^d;\R^d) \implies K_\delta[u] \in H^s(\R^d;\R^d) \text{ for every } s > 0 \, .
  \]
  Indeed, Lemma~\ref{lem:omega_delta_low_high_frequencies} implies that $\omega_\delta(\xi) \leq C(1 + |\xi|^\alpha)$. Then, being $\hat u \in \Schwartz(\R^d;\R^d)$, we have that~$(1+|\xi|^2)^\frac{s}{2} \omega_\delta(\xi) \hat{u}(\xi) \in L^2(\R^d;\R^d)$ for every~$s > 0$. By~\eqref{eq:Hs_in_Fourier_space} we deduce that~$K_\delta[u] \in H^s(\R^d;\R^d)$ for every~$s > 0$. Note that, in particular, $K_\delta[u] \in C^\infty(\R^d;\R^d)$.
\end{remark}

We are interested in the limit of the dispersion relation~$\omega_\delta$ as~$\delta \to 0$. This will allow us to understand the behavior of the peridynamics waves in the limit~$\delta \to 0$, since the result shows that~$\omega_\delta$ approaches the dispersion relation of the classical wave equation as~$\delta \to 0$.

\begin{lemma} \label{lem:difference_omega_delta_gamma_xi}
  For every~$\xi \in \R^d$, we have that
  \[ \label{eq:limit_omega_delta}
  \big|\omega_\delta(\xi) - \gamma |\xi| \big| \leq C \delta |\xi|^2 \, .
  \]
\end{lemma}
\begin{proof}
  We use the Taylor expansion of the cosine function to deduce that 
  \[
  \Big| 1 - \cos(y \cdot \xi) - \frac{1}{2} |y \cdot \xi|^2 \Big| \leq \frac{1}{4!}|y \cdot \xi|^4 \, .
  \]
  Substituting in the definition of~$\omega_\delta$ in~\eqref{eq:def_omega_delta} and by the alternative formula for~$\gamma$ derived in~\eqref{eq:gamma_alternative_formula}, we get that
  \[
    \begin{split}
      \big|\omega_\delta^2(\xi) - \gamma^2 |\xi|^2\big| & = \Big| \frac{\kappa}{\delta^{2(1-\alpha)}} \int_{\R^d} \chi_\delta(y) \frac{1 - \cos(y \cdot \xi)}{|y|^{d + 2 \alpha}} \, \d y -    \frac{\kappa}{2 \delta^{2(1-\alpha)}}  \int_{\R^d} \chi_\delta(y) \frac{|y \cdot \xi|^2}{|y|^{d + 2 \alpha}}  \, \d y \Big| \\
      & \leq \frac{\kappa}{4! \delta^{2(1-\alpha)}} \int_{\R^d} \chi_\delta(y)   \frac{|y \cdot \xi|^4}{|y|^{d + 2 \alpha}} \, \d y  \, .
    \end{split}
  \]
  To estimate the integral in the right-hand side, we apply the change of variables~$y = \delta z$ to infer that 
  \[
    \begin{split}
      & \frac{\kappa}{4! \delta^{2(1-\alpha)}} \int_{\R^d} \chi_\delta(y)   \frac{|y \cdot \xi|^4}{|y|^{d + 2 \alpha}} \, \d y  \leq |\xi|^4 \frac{\kappa}{4! \delta^{2(1-\alpha)}} \int_{\R^d} \chi\left(\frac{|y|}{\delta}\right)   \frac{1}{|y|^{d + 2 \alpha - 4}} \, \d y \\
      & = |\xi|^4 \frac{\kappa}{4! \delta^{2(1-\alpha)}} \int_{\R^d} \chi(|z|)   \frac{1}{\delta^{d + 2 \alpha - 4} |z|^{d + 2 \alpha - 4}} \delta^d \, \d z = \delta^2 |\xi|^4 \frac{\kappa}{4!} \int_{\R^d} \chi(|z|)   \frac{1}{|z|^{d + 2 \alpha - 4}}  \, \d z \\
      & \leq C \delta^2 |\xi|^4 \, , 
    \end{split}
  \]
  since $\frac{1}{|z|^{d + 2 \alpha - 4}}$ is integrable in~$B_1$. 
  Putting the estimates together, we obtain that 
  \[
    \big|\omega_\delta^2(\xi) - \gamma^2 |\xi|^2\big| \leq C \delta^2 |\xi|^4 
  \]
  and~\eqref{eq:limit_omega_delta} follows.
\end{proof}

The following technical lemma provides a bound on the derivatives of~$\omega_\delta^2$.

\begin{lemma} \label{lem:smoothness_omega_delta_2}
  The function $\xi \mapsto \omega_\delta^2(\xi)$ belongs to~$C^\infty(\R^d)$.
  Moreover, for every multi-index~$\mathbf{k}$ with~$|\mathbf{k}| \geq 1$, we have that  
    \[ \label{eq:derivative_omega_delta_square_estimate}
     | \D^\mathbf{k} \omega_\delta^2(\xi) \big| \leq C_\mathbf{k}  \, , \quad \text{for every } \xi \in \R^d \, ,
    \]
    for some constant~$C_\mathbf{k}$ depending on~$\mathbf{k}$ (and also on~$d$, $\delta$, $\alpha$, and~$\kappa$). 
\end{lemma}
\begin{proof}
  By the definition in~\eqref{eq:def_omega_delta}, we have that
  \[ 
    \omega_\delta^2(\xi) =  \frac{\kappa}{\delta^{2(1-\alpha)}} \int_{\R^d} \chi_\delta(y) \frac{1 - \cos(y \cdot \xi)}{|y|^{d + 2 \alpha}} \, \d y \, .
\]
Developing the cosine function with the Taylor expansion, for every~$\xi \in \R^d$ and for~$h \in (0,1)$ we estimate
\[ \label{eq:omega_delta_square_smoothness_bound}
  \begin{split}
    & \frac{1}{|h|}\Big| \frac{1 - \cos(y \cdot (\xi+h))}{|y|^{d + 2 \alpha}} - \frac{1 - \cos(y \cdot \xi)}{|y|^{d + 2 \alpha}} \Big| = \frac{1}{|h|}\Big| \frac{\sin(y \cdot \xi) y \cdot h + O(|y \cdot h|^2)}{|y|^{d + 2 \alpha}} \Big|  \\
    & \quad \leq \frac{|y^2| |\xi| + C |y|^2 |h|}{|y|^{d + 2 \alpha}} \leq \frac{|\xi| + C}{|y|^{d + 2 \alpha-2}} \in L^1(B_\delta) \, .
  \end{split}
\]
By the Dominated Convergence Theorem, we deduce that~$\omega_\delta^2(\xi)$ is differentiable. 
Arguing iteratively in the same way, we obtain that~$\omega_\delta^2(\xi)$ is smooth with derivatives~$\D^\mathbf{k}$ with~$\mathbf{k}$ multi-index with~$|\mathbf{k}| \geq 1$:
  \[ \label{eq:derivative_omega_delta_square}
    \D^\mathbf{k} \omega_\delta^2(\xi) =  \begin{cases} \displaystyle  
      \frac{\kappa}{\delta^{2(1-\alpha)}} \int_{\R^d} \chi_\delta(y) \frac{(-1)^{\frac{|\mathbf{k}|+1}{2}} y^{\mathbf{k}} \sin(y \cdot \xi)}{|y|^{d + 2 \alpha}} \, \d y \, , & \text{if } |\mathbf{k}| \text{ is odd} \, , \\
      \displaystyle \frac{\kappa}{\delta^{2(1-\alpha)}} \int_{\R^d} \chi_\delta(y) \frac{(-1)^{\frac{|\mathbf{k}|}{2}+1} y^\mathbf{k} \cos(y \cdot \xi)}{|y|^{d + 2 \alpha}} \, \d y \, , & \text{if } |\mathbf{k}| \text{ is even} \, , 
    \end{cases}
    \quad \text{for } \xi \in \R^d \, .
    \]
 
Let us now prove~\eqref{eq:derivative_omega_delta_square_estimate}.
It is enough to prove the estimate for~$|\xi| \geq 1$.
    Let us consider a multi-index~$\mathbf{k}$ with~$|\mathbf{k}| \geq 1$ odd. 
    In the following, we write~$\xi = |\xi| \nu$ with~$\nu \in \S^{d-1}$. 
    From~\eqref{eq:derivative_omega_delta_square} and using~\eqref{eq:properties_chi_delta}, we change variables~$z = |\xi|y$ to estimate
    \[
      \begin{split}
      | \D^\mathbf{k} \omega_\delta^2(\xi)|  & = \Big| \frac{\kappa}{\delta^{2(1-\alpha)}} \int_{\R^d} \chi_\delta(y) \frac{(-1)^{\frac{|\mathbf{k}|+1}{2}} y^{|\mathbf{k}|} \sin(y \cdot \xi)}{|y|^{d + 2 \alpha}} \, \d y  \Big| \leq   \frac{\kappa}{\delta^{2(1-\alpha)}} \int_{B_\delta}  \frac{|y|^{\mathbf{k}} |\sin(y \cdot \xi)|}{|y|^{d + 2 \alpha}} \, \d y   \\
      & =   |\xi|^{2\alpha - |\mathbf{k}|}  \frac{\kappa}{\delta^{2(1-\alpha)}} \int_{B_{|\xi|\delta}}  \frac{|z|^{|\mathbf{k}|} |\sin(z \cdot \nu)|}{|z|^{d + 2 \alpha}} \, \d z \\
      & \leq |\xi|^{2\alpha - |\mathbf{k}|}  \frac{\kappa}{\delta^{2(1-\alpha)}} \left( \int_{B_1}  \frac{1}{|z|^{d + 2 \alpha - |\mathbf{k}| - 1}} \, \d z + \int_{B_{|\xi|\delta} \sm B_1}  \frac{1}{|z|^{d + 2 \alpha - |\mathbf{k}|}} \, \d z   \right) \, .
      \end{split}
    \]  
    The first integral is finite. Integrating the second integral in spherical coordinates, we get that
    \[
      \int_{B_{|\xi|\delta} \sm B_1}  \frac{1}{|z|^{d + 2 \alpha - |\mathbf{k}|}} \, \d z = \H^{d-1}(\S^{d-1}) \int_1^{|\xi|\delta}  \rho^{-2 \alpha + |\mathbf{k}| - 1}  \, \d \rho = \frac{\H^{d-1}(\S^{d-1})}{-2 \alpha + |\mathbf{k}|} \left(|\xi\delta|^{-2 \alpha + |\mathbf{k}|} - 1 \right) \, .
    \]
    The estimate~\eqref{eq:derivative_omega_delta_square_estimate} follows.

    If~$|\mathbf{k}| \geq 2$ is even, from~\eqref{eq:derivative_omega_delta_square}, we estimate
    \[
     | \D^\mathbf{k} \omega_\delta^2(\xi)|  = \Big| \frac{\kappa}{\delta^{2(1-\alpha)}} \int_{\R^d} \chi_\delta(y) \frac{(-1)^{\frac{|\mathbf{k}|}{2}+1} y^\mathbf{k} \cos(y \cdot \xi)}{|y|^{d + 2 \alpha}} \, \d y  \Big| \leq \frac{\kappa}{\delta^{2(1-\alpha)}} \int_{B_\delta} \frac{1}{|y|^{d + 2 \alpha - |\mathbf{k}|}} \, \d y  \, ,
    \]
    and the integral is finite. 
\end{proof}
  
\subsection{Energy space} 
A nonlocal energy is naturally associated to the nonlocal operator~$K_\delta$ in the peridynamics model, given by
\[ \label{eq:nonlocal energy}
\frac{1}{2} \frac{\kappa}{\delta^{2(1-\alpha)}} \int_{\R^d} \int_{\R^d} \chi_\delta(y) \frac{|u(x) - u(x-y)|^2}{|y|^{d + 2 \alpha}} \, \d x \, \d y =  \int_{\R^d} \int_{\R^d} \Phi_\delta(y, u(x) - u(x-y)) \, \d x \, \d y \, .
\]
This nonlocal energy allows us to define the energy space naturally associated to the peridynamics model. 

\begin{definition}
  We let~$\W_\delta^\alpha$ be the space of functions~$u \in L^2(\R^d;\R^d)$ such that the nonlocal energy in~\eqref{eq:nonlocal energy} is finite, \ie,
\[
  \W_\delta^\alpha = \Big\{ u \in L^2(\R^d;\R^d) \, : \, \int_{\R^d} \int_{\R^d} \Phi_\delta(y, u(x) - u(x-y)) \, \d x \, \d y < + \infty \Big\} \, .
\]
The space~$\W_\delta^\alpha$ is endowed with the norm defined by
\[ \label{eq:norm_energy_space}
  \|u\|_{\W_\delta^\alpha} = \left( \|u\|_{L^2}^2 + [u]_{\W_\delta^\alpha}^2 \right)^\frac{1}{2} \, ,
\]
where 
\[
  [u]_{\W_\delta^\alpha} = \left( \int_{\R^d} \int_{\R^d} \Phi_\delta(y, u(x) - u(x-y)) \, \d x \, \d y \right)^\frac{1}{2} \, .
\]
We will refer to~$\W_\delta^\alpha$ as the \emph{energy space} associated to the peridynamics model.
\end{definition}

In the next result, we show that, for a given~$\delta > 0$, the space~$\W_\delta^\alpha$ is equivalent to the fractional Sobolev space~$H^\alpha(\R^d;\R^d)$.

\begin{proposition} \label{prop:equivalence_H_W}
  If~$u \in H^\alpha(\R^d;\R^d)$, then~$u \in \W_\delta^\alpha$ and 
  \[ \label{eq:inclusion_H_W}
  [u]_{\W_\delta^\alpha} \leq C \delta^{-(1-\alpha)}[u]_{H^\alpha} \,.
  \]
  Conversely, if~$u \in \W_\delta^\alpha$, then~$u \in H^\alpha(\R^d;\R^d)$ and 
  \[ \label{eq:inclusion_W_H}
  [u]_{H^\alpha} \leq C \delta^{- \alpha} \|u\|_{L^2} + C \delta^{1-\alpha} [u]_{\W_\delta^\alpha} \, .
  \]
  In~\eqref{eq:inclusion_H_W} and~\eqref{eq:inclusion_W_H}, the constant~$C$ is independent of~$\delta$ and~$\alpha$.
\end{proposition}
\begin{proof}
  The inequality~\eqref{eq:inclusion_H_W} follows directly from the definition of the norm in~$\W_\delta^\alpha$ and by the properties of~$\chi_\delta$ in~\eqref{eq:properties_chi_delta}.
  Indeed, for every~$u \in H^\alpha(\R^d;\R^d)$ we have that 
  \[ 
    \begin{split}
      [u]_{\W_\delta^\alpha}^2 = & \int_{\R^d} \int_{\R^d} \Phi_\delta(y, u(x) - u(x-y)) \, \d x \, \d y = \int_{\R^d} \int_{\R^d} \frac{1}{2} \frac{\kappa}{\delta^{2(1-\alpha)}} \chi_\delta(y) \frac{|u(x) - u(x-y)|^2}{|y|^{d + 2 \alpha}} \, \d x \, \d y  \\
      & \quad \leq \frac{1}{2} \frac{\kappa}{\delta^{2(1-\alpha)}} \int_{\R^d} \int_{\R^d}   \frac{|u(x) - u(x-y)|^2}{|y|^{d + 2 \alpha}} \, \d x \, \d y =   \frac{1}{2} \frac{\kappa}{\delta^{2(1-\alpha)}} [u]_{H^\alpha}^2 \, .
    \end{split}
  \]

  Let us prove~\eqref{eq:inclusion_W_H}. 
  By~\eqref{eq:properties_chi_delta} we have that 
  \[
    \begin{split}
      [u]_{H^\alpha}^2 & = \int_{\R^d} \int_{\R^d} \frac{|u(x) - u(x-y)|^2}{|y|^{d + 2 \alpha}} \, \d x \, \d y  \\
      & = \int_{\R^d} \int_{\R^d} \chi_\delta(y) \frac{|u(x) - u(x-y)|^2}{|y|^{d + 2 \alpha}} \, \d x \, \d y + \int_{\R^d} \int_{\R^d} (1-\chi_\delta(y)) \frac{|u(x) - u(x-y)|^2}{|y|^{d + 2 \alpha}} \, \d x \, \d y \\
      & \leq \frac{2 \delta^{2(1-\alpha)}}{\kappa} [u]_{\W_\delta^\alpha}^2   + 2 \int_{\R^d}  \int_{\R^d} \Big( |u(x)|^2 + |u(x-y)|^2 \Big) \frac{1-\chi_\delta(y)}{|y|^{d + 2 \alpha}} \, \d x \, \d y \\ 
      &  \leq \frac{2 \delta^{2(1-\alpha)}}{\kappa} [u]_{\W_\delta^\alpha}^2 + 4 \|u\|_{L^2}^2 \int_{\R^d \sm B_{\delta/2}}   \frac{1-\chi_\delta(y)}{|y|^{d + 2 \alpha}} \,   \d y \, .
    \end{split}
  \]
  By applying the change of variables~$y = \delta z$, we obtain that 
  \[
    \int_{\R^d \sm B_{\delta/2}}   \frac{1-\chi_\delta(y)}{|y|^{d + 2 \alpha}} \,   \d y = \delta^{-2\alpha} \int_{\R^d \sm B_{1/2}}   \frac{1-\chi(z)}{|z|^{d + 2 \alpha}} \,   \d z \leq C \delta^{-2 \alpha} \, ,
  \]
  hence the thesis.
\end{proof}

Thanks to the symbol~$\omega_\delta^2(\xi)$ of the peridynamics operator defined in~\eqref{eq:def_omega_delta}, we can characterize the energy space~$\W_\delta^\alpha$ in terms of the Fourier transform, as in~\eqref{eq:Hs_in_Fourier_space}. We prove this result in two steps contained in the next two lemmata.

\begin{lemma} \label{lem:Kuu}
  For every~$u \in \Schwartz(\R^d;\R^d)$, we have that
  \[
    \langle - K_\delta[u],u \rangle = [u]_{\W_\delta^\alpha}^2   \, .
  \]
\end{lemma}

\begin{proof}
  By the definition of~$K_\delta[u]$, we have that
  \[ \label{eq:Kuu}
    \begin{split}
      - \langle K_\delta[u],u \rangle & = - \int_{\R^d} K_\delta[u](x) \cdot u(x) \, \d x \\
      &  = \frac{\kappa}{\delta^{2(1-\alpha)}} \int_{\R^d} \left(
       \pvint_{\R^d} \chi_\delta(y) \frac{u(x) - u(x-y)}{|y|^{d + 2 \alpha}}  \, \d y \right) \cdot u(x) \, \d x \\
      &  = \frac{\kappa}{\delta^{2(1-\alpha)}} \int_{\R^d} \left(
       \lim_{\e \to 0}\int_{\R^d \sm B_\e} \chi_\delta(y) \frac{u(x) - u(x-y)}{|y|^{d + 2 \alpha}}  \, \d y \right)\cdot u(x) \, \d x \, .
    \end{split}
  \]
  We apply the radial symmetry and the bound deduced in~\eqref{eq:u_expansion_in_L1} to estimate 
  \[
    \begin{split}
      & \Big|\int_{\R^d \sm B_\e} \chi_\delta(y) \frac{u(x) - u(x-y)}{|y|^{d + 2 \alpha}}  \, \d y\Big| = \frac{1}{2}\Big|\int_{\R^d \sm B_\e} \chi_\delta(y) \frac{u(x+y) + u(x-y) - 2 u(x)}{|y|^{d + 2 \alpha}}  \, \d y\Big| \\
      & \quad \leq \int_{\R^d \sm B_\e} C_\delta \frac{1}{1+|x|^{d+1}} \frac{1}{|y|^{d + 2 \alpha-2}}  \, \d y \leq C'_\delta \frac{1}{1+|x|^{d+1}} \in L^1(\R^d)\, . 
    \end{split}
  \]
  This allows us to apply the Dominated Convergence Theorem in~\eqref{eq:Kuu} and then Fubini's Theorem to infer that
  \[ \label{eq:Kuu_2}
    \begin{split}
      & \langle - K_\delta[u],u \rangle   =  \frac{\kappa}{\delta^{2(1-\alpha)}} \lim_{\e \to 0}\int_{\R^d} \left( \int_{\R^d \sm B_\e} \chi_\delta(y) \frac{u(x) - u(x-y)}{|y|^{d + 2 \alpha}}  \, \d y \right)\cdot u(x) \, \d x  \\
      & \quad =  \frac{\kappa}{\delta^{2(1-\alpha)}} \lim_{\e \to 0} \int_{\R^d \sm B_\e} \frac{\chi_\delta(y)}{|y|^{d + 2 \alpha}} \left( \int_{\R^d} \big( u(x) - u(x-y) \big) \cdot u(x) \, \d x \right) \, \d y \, .
    \end{split}
  \]
  A change of variables yields 
  \[
    \frac{1}{2}\int_{\R^d} \big( u(x) - u(x-y) \big) \cdot u(x) \, \d x  = - \frac{1}{2}\int_{\R^d} \big( u(x) - u(x+y) \big) \cdot u(x+y) \, \d x \,.
  \]
  By symmetry in the integral in the~$y$ variable, from the previous equality and~\eqref{eq:Kuu_2} we deduce that 
  \[
    \begin{split}
      \langle -K_\delta[u],u \rangle  & =  \frac{\kappa}{\delta^{2(1-\alpha)}} \lim_{\e \to 0} \int_{\R^d \sm B_\e} \frac{\chi_\delta(y)}{|y|^{d + 2 \alpha}} \left(
       \int_{\R^d}   \frac{1}{2}\big( u(x) - u(x-y) \big) \cdot \big( u(x) -  u(x-y) \big)  \, \d x \right) \, \d y \\
      & = \frac{1}{2}\frac{\kappa}{\delta^{2(1-\alpha)}}  \int_{\R^d} \int_{\R^d} \chi_\delta(y) \frac{|u(x) -  u(x-y)|^2}{|y|^{d + 2 \alpha}}   \, \d x \, \d y = [u]_{\W_\delta^\alpha}^2 \, ,
    \end{split}
  \]
  where we applied the integrability of $\chi_\delta(y) \frac{|u(x) -  u(x-y)|^2}{|y|^{d + 2 \alpha}}$ in~$\R^d \x \R^d$ since $u \in \Schwartz(\R^d;\R^d) \subset H^\alpha(\R^d;\R^d)$.
\end{proof}

\begin{lemma}  \label{lem:W_in_Fourier_space}
  For every~$u \in \W_\delta^\alpha$, we have that 
  \[ \label{eq:norm_W_and_omega}
  [u]_{W_\delta^\alpha} = \| \omega_\delta \hat u \|_{L^2} \, , \quad   \|u\|_{\W_\delta^\alpha} = \|(1+\omega_\delta^2)^\frac{1}{2} \hat{u}\|_{L^2} \, .
  \]
  In particular, 
  \[ \label{eq:W_in_Fourier_space}
  \W_\delta^\alpha = \{ u \in L^2(\R^d;\R^d) \ : \ (1+\omega_\delta^2(\xi))^\frac{1}{2} \hat{u}(\xi) \in L^2(\R^d;\R^d) \} \, .
  \]
\end{lemma}

\begin{proof}
  Let~$u \in \W_\delta^\alpha$ or, equivalently, $u \in H^\alpha(\R^d;\R^d)$ by Proposition~\ref{prop:equivalence_H_W}.
  Let~$u_n \in \Schwartz(\R^d;\R^d)$ be a sequence such that~$\|u_n - u\|_{H^\alpha} \to 0$ as~$n \to +\infty$.
  By the definition of~$\|\cdot\|_{\W_\delta^\alpha}$ and by Lemma~\ref{lem:Kuu}, we deduce that 
  \[
  \|u_n\|_{\W_\delta^\alpha}^2 = \|u_n\|_{L^2}^2 + [u_n]_{\W_\delta^\alpha}^2 = \|u_n\|_{L^2}^2 + \langle - K_\delta[u_n],u_n \rangle = \|u_n\|_{L^2}^2   - \int_{\R^d} K_\delta[u_n](x) \cdot u_n(x) \, \d x \, .
  \]
  By Remark~\ref{rem:regularity_of_K_delta} we have, in particular, that~$K_\delta[u_n] \in L^2(\R^d;\R^d)$. We are thus in a position to apply Plancherel's Theorem. Exploiting Lemma~\ref{lem:symbol_K_delta}, we infer that
  \[
    \begin{split}  
      \|u_n\|_{\W_\delta^\alpha}^2 & =  \|\hat{u}_n\|_{L^2}^2   - \int_{\R^d} \hat{K_\delta[u_n]}(\xi) \cdot \overline{\hat{u}_n(\xi)} \, \d \xi  = \int_{\R^d}   |\hat{u}_n(\xi)|^2 \, \d \xi   + \int_{\R^d} \omega_\delta^2(\xi) |\hat{u}_n(\xi)|^2 \, \d \xi  \\
      &\qquad = \|(1+\omega_\delta^2)^\frac{1}{2} \hat{u}_n\|_{L^2}^2 \, . 
    \end{split}
  \]
  On the one hand, by Proposition~\ref{prop:equivalence_H_W} we deduce that~$\|u_n\|_{\W_\delta^\alpha} \to \|u\|_{\W_\delta^\alpha}$ as~$n \to +\infty$. 
  On the other hand, Lemma~\ref{lem:omega_delta_low_high_frequencies} implies that~$\omega_\delta(\xi) \leq C(1 + |\xi|^\alpha)$, hence, by~\eqref{eq:Hs_norm_equivalence},
  \[
    \|(1+\omega_\delta^2)^\frac{1}{2} (\hat{u}_n - \hat{u})\|_{L^2} \leq C \|(1+|\xi|^2)^\frac{\alpha}{2} (\hat{u}_n - \hat{u})\|_{L^2} \leq C\|u_n - u\|_{H^\alpha} \to 0 \, .
  \]
  This concludes the proof.
\end{proof}

\section{Well-posedness of the linear peridynamics model with low regularity}
\label{67r483ebcnx98765yhbnetuku}

In this section we present the linear peridynamics model and we provide some well-posedness results. 

\subsection{The linear peridynamics model} 
The linear peridynamics model is a nonlocal model for the evolution of a displacement field~$u \colon [0,+\infty) \x \R^d \to \R^d$, governed by the initial-value problem in~\eqref{eq:linear_peridynamics}.

The integro-differential equation in~\eqref{eq:linear_peridynamics} can be interpreted as a nonlocal version of the conservation of linear momentum equation in classical elasticity theory for a body with mass density~$\rho = 1$ and subjected to no external forces. 
The derivation of this nonlocal evolution model is not based on the presence of the classical contact forces in the material, but on the assumption that the force field source at a point~$x$ depends on the displacement computed on all the points in the region in the peridynamics horizon of~$x$, \ie, in~$B_\delta(x)$. This dependence is modeled by the nonlocal operator~$K_\delta$ through the singular interaction kernel~$f_\delta$.

In the next sections we will provide some well-posedness results for the linear peridynamics model in~\eqref{eq:linear_peridynamics} and we will compare the solutions to the linear peridynamics model with the solutions to the classical wave equation. Well-posedness results for the peridynamics model have been already studied in the literature. See~\cite{CocDipMadVal18} for existence and uniqueness results for the nonlinear peridynamics model in the whole space~$\R^d$ for solutions belonging to the energy space. 

\subsection{Total energy} An energy is naturally associated to solutions of the linear peridynamics model in~\eqref{eq:linear_peridynamics}. For a solution~$u \in L^\infty((0,+\infty);H^\alpha(\R^d;\R^d)))$ with~$\de_t u \in L^\infty((0,+\infty);L^2(\R^d;\R^d))$, it is given for a.e.~$t$ by the sum of the kinetic energy and the nonlocal potential energy, \ie,
\[ \label{eq:total_energy}
\begin{split}
E^{K_\delta}(t) & = \frac{1}{2} \int_{\R^d}   |\de_t u(t,x)|^2 \, \d x + \frac{1}{2} \int_{\R^d} \int_{\R^d} \Phi_\delta(y,u(t,x) - u(t,x-y)) \, \d x \, \d y \\
& = \frac{1}{2} \int_{\R^d}   |\de_t u(t,x)|^2 \, \d x + \frac{1}{4} \frac{\kappa}{\delta^{2(1-\alpha)}}\int_{\R^d} \int_{\R^d} \chi_\delta(y) \frac{|u(t,x) - u(t,x-y)|^2}{|y|^{d + 2 \alpha}} \, \d x \, \d y \\
& = \frac{1}{2} \int_{\R^d}   |\de_t u(t,x)|^2 \, \d x + \frac{1}{2}[u(t,\cdot)]^2_{\W_\delta^\alpha} \, .
\end{split}
  \]

In this paper we will focus on solutions with low spatial regularity, in particular on solutions that do not necessarily belong to the energy space.

In the results proved in this section it will become clear that the natural assumptions on the initial data are such that the initial velocity~$v_0$ has a lower regularity than the initial displacement~$u_0$. More precisely, it loses~$\alpha$ derivatives with respect to~$u_0$.

\subsection{Existence} 

First of all, let us provide the definition of distributional solution to the linear peridynamics model in~\eqref{eq:linear_peridynamics}.

\begin{definition} \label{def:distributional_solution}
  Let~$u_0 \in \Dcal'(\R^d;\R^d)$ and~$v_0 \in \Dcal'(\R^d;\R^d)$. A distributional solution to the linear peridynamics model in~\eqref{eq:linear_peridynamics} with initial data~$(u_0,v_0)$ is a distribution~$u \in \Dcal'(\R \x \R^d;\R^d)$ such that $\supp(u) \subset [0,+\infty) \x \R^d$ and
  \[ \label{eq:distributional_solution}
  \langle u, \de_{tt} \varphi - K_\delta^*[\varphi] \rangle =  - \langle u_0, \de_t \varphi(0,\cdot) \rangle + \langle v_0, \varphi(0,\cdot) \rangle \, ,    
  \]
  for every~$\varphi \in \Dcal(\R \x \R^d;\R^d)$.
  
In the above formula, the duality pairing in the left-hand side is between $\Dcal'(\R \x \R^d;\R^d)$ and $\Dcal(\R \x \R^d;\R^d)$, while the duality pairings in the right-hand side are between $\Dcal'(\R^d;\R^d)$ and $\Dcal(\R^d;\R^d)$.
\end{definition}

The following technical lemma will be used in the proof of Theorem~\ref{thm:existence}. 
The reader interested in the main content of this section can skip the tedious computations in the proof.

\begin{lemma} \label{lem:smoothness_cos_sin}
  Let~$t > 0$. The functions 
  \[
    \xi \mapsto \cos(\omega_\delta(\xi) t) \, , \quad \xi \mapsto \frac{\sin(\omega_\delta(\xi) t)}{\omega_\delta(\xi)}  \, ,  \quad \xi \mapsto \omega_\delta(\xi) \sin(\omega_\delta(\xi) t) 
  \]
  belong to~$C^\infty(\R^d)$ and their derivatives of any order are bounded (by a constant depending on the order of differentiation).
\end{lemma} 

\begin{proof}
  We provide the details for the function~$\xi \mapsto \cos(\omega_\delta(\xi) t)$, as the argument for~$\frac{\sin(\omega_\delta(\xi) t)}{\omega_\delta(\xi)} $ is completely analogous.
  We write~$\cos(\omega_\delta(\xi) t) = f(\omega_\delta^2(\xi))$ where~$f(\zeta) = \cos(\sqrt{\zeta t^2})$ for~$\zeta \geq 0$ and, using the Taylor series of the cosine function, we observe that~$f$ can be extended to an analytic function on~$\R$. 
  Moreover, the derivatives of~$f$ are bounded (by constants depending on the order of differentiation).
  Indeed, for~$|\zeta| \leq 1$ this follows by the smoothness of~$f$.
  For~$|\zeta| > 1$ one argues by induction to deduce that, for~$k \geq 1$, the~$k$-th derivative of~$f$ has the form
  \[ \label{eq:derivative_cos}
  \frac{\d^k f}{\d \zeta^k}(\zeta) = \sum_{h=1}^k \frac{\d^h}{\d \theta^h} \cos(\theta)\Big|_{\theta = \sqrt{\zeta t^2}}  \, B_{k,h}\left(
   \frac{\d}{\d \zeta} \big( \sqrt{\zeta t^2} \big), \dots, \frac{\d^{k-h+1}}{\d \zeta^{k-h+1}} \big( \sqrt{\zeta t^2} \big) \right) \, ,
  \]
  where~$B_{k,h}$ are suitable polynomials. 
  The specific expression of these polynomials\footnote{In fact, these polynomials can be computed explicitly: they are the partial exponential Bell polynomials. They appear in the Faà di Bruno's formula for the chain rule with high-order differentiation~\cite{Joh02}.} is not relevant, as it is enough to observe that, for~$|\zeta| > 1$,
all the derivatives of~$\sqrt{\zeta t^2}$ are bounded by a constant (depending on the order of differentiation).

  Let us now turn to the partial derivatives of~$\cos(\omega_\delta(\xi) t) = f(\omega_\delta^2(\xi))$.
  As above, one gets that the derivatives corresponding to a multi-index~$\mathbf{k}$ with~$|\mathbf{k}| \geq 1$ are of the form
  \[  
  \D^\mathbf{k}_\xi \big(  \cos(\omega_\delta(\xi) t) \big) =  \D^\mathbf{k}_\xi \big(  f(\omega_\delta^2(\xi)) \big) = \sum_{|\mathbf{j}| \leq |\mathbf{k}|} \frac{\d^{|\mathbf{j}|} f}{\d \zeta^{|\mathbf{j}|}}(\zeta)\Big|_{\zeta = \omega_\delta^2(\xi)} \,  B_{\mathbf{k},\mathbf{j}}\Big( \big(\D^{\mathbf{i}}_\xi ( \omega_\delta^2(\xi) ) \big)_{\mathbf{i}}\Big) \, ,
  \]
  where the sum runs over all multi-indices~$\mathbf{j}$ such that~$|\mathbf{j}| \leq |\mathbf{k}|$ and~$B_{\mathbf{k},\mathbf{j}}$ are suitable polynomials applied to the partial derivatives of~$\omega_\delta^2(\xi)$ up to a certain order. 
  The specific expression is not relevant for our\footnote{The interested reader can find and explicit formula for this generalization of Faà di Bruno's formula in~\cite{Sch19}.} purposes.
  What is relevant is that the derivatives~$\frac{\d^{|\mathbf{j}|} f}{\d \zeta^{|\mathbf{j}|}}(\zeta)$ are bounded and by Lemma~\ref{lem:smoothness_omega_delta_2} we have that the partial derivatives of~$\omega_\delta^2(\xi)$ (hence their composition with polynomials) are bounded.

  The argument for the function~$\xi \mapsto \omega_\delta(\xi) \sin(\omega_\delta(\xi) t)$ is adapted from the above with some modifications. 
  We observe that~$\omega_\delta(\xi) \sin(\omega_\delta(\xi) t) = g(\omega_\delta^2(\xi))$ where~$g(\zeta) = \sqrt{\zeta} \sin(\sqrt{\zeta t^2})$ for~$\zeta \geq 0$. 
  The function~$g$ is analytic.
  The derivatives of~$g$ are bounded, as can be seen by a formula analogous to~\eqref{eq:derivative_cos} with~$\theta \sin(\theta)$ in place of~$\cos(\theta)$.
  The rest of the argument is then analogous to the one for~$\cos(\omega_\delta(\xi) t)$.
\end{proof}

We now establish the main existence result of the paper. 

\begin{proof}[Proof of Theorem~\ref{thm:existence}]
  \underline{\emph{Definition of $u_\delta$}}.  We observe that the right-hand side of~\eqref{eq:Fourier_solution} is a tempered distribution for every~$t \in [0,+\infty)$. Indeed, by Lemma~\ref{lem:smoothness_cos_sin}, we have that~$\cos(\omega_\delta(\cdot) t) \hat{u}_0$ is the product of a smooth function with bounded derivatives of any order and a tempered distribution, hence it is a tempered distribution. 
  The same holds for~$\frac{\sin(\omega_\delta(\cdot) t)}{\omega_\delta(\cdot)}\hat{v}_0$. 
  
For every~$t \in [0,+\infty)$, we define~$u_\delta(t,\cdot) \in \Schwartz'(\R^d;\C^d)$, introduced via inverse Fourier transform as
  \[
    u_\delta(t,\cdot) := \F^{-1}\Big[\cos(\omega_\delta(\cdot) t) \hat{u}_0  + \frac{\sin(\omega_\delta(\cdot) t)}{\omega_\delta(\cdot)} \hat{v}_0\Big] \, ,
  \]
  so that~\eqref{eq:Fourier_solution} is satisfied. 
  We observe that~$u_\delta$ is real-valued by noticing that~$\hat{u}_\delta$ is Hermitian, see~\eqref{rem:omega_delta_radial}. 
  This follows from the fact that~$\hat{u}_0$ and~$\hat{v}_0$ are real-valued, thus~$\hat{u}_0$ and~$\hat{v}_0$ are Hermitian, and by the symmetry of~$\omega_\delta$ (see Remark~\ref{rem:omega_delta_radial}).

  Next, we observe that~$u_\delta(t,\cdot) \in H^s(\R^d;\R^d)$ for every~$t \in [0,+\infty)$. 
  Indeed, by the characterization of~$H^s(\R^d;\R^d)$ in terms of the Fourier transform, we have that
  \[
    (1+|\xi|^2)^\frac{s}{2} \hat u_0(\xi) \in L^2(\R^d;\C^d) \, , \quad (1+|\xi|^2)^\frac{s-\alpha}{2} \hat v_0(\xi) \in L^2(\R^d;\C^d) \, .
  \]
  Hence,
  \[
    (1+|\xi|^2)^\frac{s}{2} \cos(\omega_\delta(\xi) t) \hat u_0(\xi) \in L^2(\R^d;\C^d) \, .
  \]
   As for the second term, we apply~\eqref{eq:omega_delta_high_frequencies} to obtain the estimates
  \[ \label{eq:estimate_high_low_frequencies}
    \begin{cases}
      \displaystyle \frac{(1+|\xi|^2)^\frac{\alpha}{2}}{\omega_\delta(\xi)} \leq C \, , &  \text{for } |\xi| > 1 \, , \\
      \displaystyle (1+|\xi|^2)^\frac{\alpha}{2} \leq 2^\frac{\alpha}{2} \leq C \, , & \text{for } |\xi| \leq 1 \, .
    \end{cases}
      \]
  and deduce that 
  \[ 
    \begin{split}
      & \Big| (1+|\xi|^2)^\frac{s}{2} \frac{\sin(\omega_\delta(\xi) t)}{\omega_\delta(\xi)} \hat v_0(\xi)\Big|  = \Big| (1+|\xi|^2)^\frac{\alpha}{2} \frac{\sin(\omega_\delta(\xi) t)}{\omega_\delta(\xi)} (1+|\xi|^2)^\frac{s-\alpha}{2} \hat v_0(\xi)\Big| \\
      & \quad \leq \Big| (1+|\xi|^2)^\frac{\alpha}{2} \frac{\sin(\omega_\delta(\xi) t)}{\omega_\delta(\xi)} (1+|\xi|^2)^\frac{s-\alpha}{2} \hat v_0(\xi) \mathds{1}_{\{|\xi| \leq 1\}}(\xi)\Big| \\
      & \quad \quad + \Big| \frac{(1+|\xi|^2)^\frac{\alpha}{2}}{\omega_\delta(\xi)}  \sin(\omega_\delta(\xi) t) (1+|\xi|^2)^\frac{s-\alpha}{2} \hat v_0(\xi) \mathds{1}_{\{|\xi| > 1\}}(\xi)\Big| \\
      & \quad \leq C t (1+|\xi|^2)^\frac{s-\alpha}{2} |\hat v_0(\xi)|\mathds{1}_{\{|\xi| \leq 1\}}(\xi) + C (1+|\xi|^2)^\frac{s-\alpha}{2} |\hat v_0(\xi)|\mathds{1}_{\{|\xi| > 1\}}(\xi) \\
      & \quad \leq C (1+t) (1+|\xi|^2)^\frac{s-\alpha}{2} |\hat v_0(\xi)| \in L^2(\R^d)\, . 
    \end{split}
  \]
  This shows that~$u_\delta(t,\cdot) \in H^s(\R^d;\R^d)$.

  Let us prove that~$u_\delta \in C([0,+\infty);H^s(\R^d;\R^d))$.  
  Let~$t$, $s \in [0,+\infty)$. 
  We compute, exploiting~\eqref{eq:estimate_high_low_frequencies}, that
  \[ \label{eq:continuity_Hs}
    \begin{split}
      & \|u_\delta(t,\cdot) - u_\delta(s,\cdot)\|_{H^s} \leq C \|(1+|\xi|^2)^{\frac{s}{2}} (\hat u_\delta(t,\xi)  - \hat u_\delta(s,\xi)) \|_{L^2} \\
      & \leq C \| ( \cos(\omega_\delta(\xi)t) - \cos(\omega_\delta(\xi)s) ) (1+|\xi|^2)^{\frac{s}{2}} \hat u_0(\xi) \|_{L^2} \\
      & \quad + C \left\| \left( \frac{\sin(\omega_\delta(\xi)t)}{\omega_\delta(\xi)} - \frac{\sin(\omega_\delta(\xi)s)}{\omega_\delta(\xi)} \right) (1+|\xi|^2)^\frac{\alpha}{2} (1+|\xi|^2)^{\frac{s-\alpha}{2}} \hat v_0(\xi)  \mathds{1}_{\{|\xi| \leq 1\}}(\xi) \right\|_{L^2} \\ 
      & \quad + C \Big\| ( \sin(\omega_\delta(\xi)t) - \sin(\omega_\delta(\xi)s) ) \frac{(1+|\xi|^2)^{\frac{\alpha}{2}}}{\omega_\delta(\xi)} (1+|\xi|^2)^{\frac{s-\alpha}{2}} \hat v_0(\xi)  \mathds{1}_{\{|\xi| > 1\}}(\xi) \Big\|_{L^2} \, ,
    \end{split}
  \]
  which converges to~$0$ as~$s\to t$,
thanks to the Dominated Convergence Theorem. 

  Let us prove that~$u_\delta \in C^1((0,+\infty);H^{s-\alpha}(\R^d;\R^d))$.
  Reasoning as above and by Lemma~\ref{lem:smoothness_cos_sin}, we obtain that 
  \[
  - \omega_\delta(\cdot)\sin(\omega_\delta(\cdot) t) \hat u_0 + \cos(\omega_\delta(\cdot) t) \hat v_0 
  \]
  is a real-valued tempered distribution for every~$t \in [0,+\infty)$ and 
  \[
  v_\delta(t,\cdot) = \F^{-1}\Big[- \omega_\delta(\cdot)\sin(\omega_\delta(\cdot) t) \hat u_0 + \cos(\omega_\delta(\cdot) t) \hat v_0\Big] \in H^{s-\alpha}(\R^d;\R^d) \, .
  \]
  We obtain that
  \[
  \begin{split}
    & \Big\|\frac{ u_\delta(t,\cdot) - u_\delta(s,\cdot)}{t-s} - v_\delta(t,\cdot) \Big\|_{H^{s-\alpha}} \\
    & \quad \leq C \left\| \left( \frac{\cos(\omega_\delta(\xi)t) - \cos(\omega_\delta(\xi)s)}{t-s} + \omega_\delta(\xi) \sin(\omega_\delta(\xi)t) \right) (1+|\xi|^2)^{\frac{s-\alpha}{2}} \hat u_0(\xi) \right\|_{L^2} \\
    & \quad \leq C \left\| \left( \frac{\sin(\omega_\delta(\xi)t) - \sin(\omega_\delta(\xi)s)}{\omega_\delta(\xi)(t-s)} - \cos(\omega_\delta(\xi)t) \right) (1+|\xi|^2)^{\frac{s-\alpha}{2}} \hat v_0(\xi) \right\|_{L^2} \, ,
  \end{split}  
  \]
  which converges to~$0$ as~$s\to t$. Here above, we have
  applied the Dominated Convergence Theorem by estimating, using Lemma~\ref{lem:omega_delta_low_high_frequencies},
  \[
    \begin{split}
      & \Big| \frac{\cos(\omega_\delta(\xi)t) - \cos(\omega_\delta(\xi)s)}{t-s} + \omega_\delta(\xi) \sin(\omega_\delta(\xi)t) \Big|^2 (1+|\xi|^2)^{s - \alpha} |\hat u_0(\xi)|^2 \\
      & \quad \leq \omega_\delta^2(\xi) (1+|\xi|^2)^{s - \alpha} |\hat u_0(\xi)|^2  \leq C (1+|\xi|^2)^{s} |\hat u_0(\xi)|^2 \in L^1(\R^d) \, ,
    \end{split}
   \] 
   and
   \[
    \begin{split}
      \Big| \frac{\sin(\omega_\delta(\xi)t) - \sin(\omega_\delta(\xi)s)}{\omega_\delta(\xi)(t-s)}  - \cos(\omega_\delta(\xi)t) \Big|^2 (1+|\xi|^2)^{s - \alpha} |\hat v_0(\xi)|^2  \leq (1+|\xi|^2)^{s - \alpha} |\hat v_0(\xi)|^2 \in L^1(\R^d) \, .
    \end{split}
   \]

  \underline{\emph{Distributional solution}}.
  We regard~$u_\delta$ as a distribution in~$\Dcal'(\R \x \R^d;\R^d)$ by setting, if the differentiation order satisfies~$s \geq 0$,
  \[
  \langle u_\delta, \varphi \rangle_{\Dcal',\Dcal} = \int_{(0,+\infty)} \int_{\R^d} u_\delta(t,x) \cdot \varphi(t,x)  \, \d x \, \d t \, , \quad \text{for every } \varphi \in \Dcal(\R \x \R^d;\R^d) \, ,
  \]
  and, if~$s < 0$,
  \[
  \langle u_\delta, \varphi \rangle_{\Dcal',\Dcal} = \int_{(0,+\infty)} \langle u_\delta(t,\cdot), \varphi(t,\cdot) \rangle_{H^{-|s|}, H^{|s|}} \, \d t \, , \quad \text{for every } \varphi \in \Dcal(\R \x \R^d;\R^d) \, .
  \] 
  
  Let us prove that~$u_\delta$ is a distributional solution to the linear peridynamics model in~\eqref{eq:linear_peridynamics} with initial data~$(u_0,v_0)$ according to Definition~\ref{def:distributional_solution}. Let~$\varphi \in \Dcal(\R \x \R^d;\R^d)$. 
  We exploit the linearity of the equation and we argue by approximation by considering a family of smooth functions~$u_\delta * \eta_\e$, where~$(\eta_\e)_\e$ is a family of mollifiers~$\eta_\e \in C^\infty_c(\R^d)$ such that~$\eta_\e \geq 0$, $\int_{\R^d} \eta_\e = 1$ and $\supp(\eta_\e) \subset B_\e$. 
  We observe that the Fourier transform of~$u_\delta * \eta_\e$ is given by (in the formula, $\hat u_0\hat \eta_\e$ and~$\hat v_0\hat \eta_\e$ are intended as the multiplication of a tempered distribution with a rapidly decreasing function)
  \[
    \begin{split}
      \hat{u_\delta * \eta_\e}(t,\cdot) & = \cos(\omega_\delta(\cdot) t) \hat u_0\hat \eta_\e + \frac{\sin(\omega_\delta(\cdot) t)}{\omega_\delta(\cdot)} \hat v_0 \hat \eta_\e \\
      & =  \cos(\omega_\delta(\cdot) t) \hat{u_0*\eta_\e} + \frac{\sin(\omega_\delta(\cdot) t)}{\omega_\delta(\cdot)} \hat{v_0*\eta_\e} \, ,
    \end{split}
  \]
  where~$u_0*\eta_\e$ and~$v_0*\eta_\e$ are smooth functions.
  This implies that~$u_\delta * \eta_\e$ is a classical solution to the problem 
  \[
  \begin{cases}
    \de_{tt} (u_\delta * \eta_\e) - K_\delta[u_\delta * \eta_\e] = 0 \, , & (t,x) \in (0,+\infty) \x \R^d \, , \\
    (u_\delta * \eta_\e)(0,x) = (u_0*\eta_\e)(x) \, , \quad \de_t (u_\delta * \eta_\e)(0,x) = (v_0*\eta_\e)(x) \, , & x \in \R^d \, .
  \end{cases}
  \]
  Integrating by parts, one gets that~$u_\delta * \eta_\e$ is a distributional solution with initial data~$(u_0*\eta_\e, v_0*\eta_\e)$ in the sense of Definition~\ref{def:distributional_solution}, \ie, for every~$\varphi \in \Dcal(\R \x \R^d;\R^d)$,
  \[
  \begin{split}
    & \langle u_\delta * \eta_\e, \de_{tt} \varphi - K_\delta^*[\varphi] \rangle =  - \langle u_0*\eta_\e, \de_t \varphi(0,\cdot) \rangle + \langle v_0*\eta_\e, \varphi(0,\cdot) \rangle \, .
  \end{split}
  \]
  Letting~$\e \to 0$, we exploit the convergence of~$u_\delta * \eta_\e$ to~$u_\delta$ in $\Dcal'(\R \x \R^d;\R^d)$ and of~$u_0*\eta_\e$ and~$v_0*\eta_\e$ to~$u_0$ and~$v_0$ in $\Dcal'(\R^d;\R^d)$ to deduce that~$u_\delta$ is a distributional solution to the linear peridynamics model in~\eqref{eq:linear_peridynamics} with initial data~$(u_0,v_0)$ in the sense of Definition~\ref{def:distributional_solution}.
  This concludes the proof.
\end{proof}

\subsection{Uniqueness} In this subsection we provide a uniqueness result for the linear peridynamics model under low regularity assumptions. 
By the linearity of the equation, we can restrict our attention to the uniqueness of the solutions with zero initial data.

We stress that, differently from~\cite[Theorem~2.5]{CocDipMadVal18}, the following uniqueness result holds for solutions that do not necessarily belong to the energy space. 
In principle, one cannot exclude that the zero initial datum evolves in a nontrivial solution belonging to a space different with less regularity than the energy space. We show that this is not the case
and thus establish the uniqueness result in Theorem~\ref{thm:uniqueness}.

\begin{proof}[Proof of Theorem~\ref{thm:uniqueness}]
  Let~$\psi \in \Dcal(\R \x \R^d;\R^d)$. 
  Let~$T > 0$ be such that $\supp(\psi) \subset (-\infty,T] \x \R^d$.
  We need to show that $\langle u_\delta , \psi \rangle = 0$.
  We will do this by arguing by duality, \ie, by building (and correcting) a smooth function~$\varphi$ such that~$\de_{tt} \varphi - K_\delta^*[\varphi] = \psi$ and use the Definition~\ref{def:distributional_solution} of distributional solution. We provide the details below. 

  \underline{\itshape Solving $\de_{tt} \varphi - K_\delta^*[\varphi] = \psi$}. We consider a solution~$\varphi$ to the final-time dual non-homogeneous problem 
  \[
  \begin{cases} 
    \de_{tt} \varphi(t,x) - K_\delta^*[\varphi(t,\cdot)](x) = \psi(t,x) \, , &   (t,x) \in (0,T) \x \R^d \, , \\
    \varphi(T,x) = 0 \, , \quad \de_t \varphi(T,x) = 0 \, , &  x \in \R^d \, .
  \end{cases} 
  \]
  Such a~$\varphi$ can be exhibited explicitly following Duhamel's principle, \ie, 
  \[
  \varphi(t,x) = -\int_t^T \tilde \varphi(t,x;\tau) \, \d \tau \, , \quad t  \in \R \, , \ x \in \R^d \, ,
  \]
  where 
  \[
  \tilde \varphi(t,x;\tau) := \frac{1}{(2\pi)^{d/2}} \int_{\R^d} \frac{\sin(\omega_\delta(\xi)(t-\tau))}{\omega_\delta(\xi)} \hat \psi(\tau,\xi) e^{i \xi \cdot x} \, \d \xi \, , \quad t \in \R \, , \ x \in \R^d \, , \ \tau \in \R \, ,
  \]  
  \ie, $\tilde \varphi(t,x;\tau)$ is a solution to the homogeneous problem
  \[
   \begin{cases}
    \de_{tt} \tilde \varphi(t,x;\tau) - K_\delta^*[\tilde \varphi(t,\cdot;\tau)](x) = 0 \, , & t \in \R \, , \ x \in \R^d \, , \\ 
    \tilde \varphi(\tau,x;\tau) = 0 \, , & x \in \R^d \, , \\ 
    \de_t \tilde \varphi(\tau,x;\tau) = \psi(\tau,x) \, , & x \in \R^d \, . 
   \end{cases}
  \]

  \underline{Estimates on $\varphi$ and its derivatives}.
  We observe that~$\varphi \in C^\infty(\R \x \R^d;\R^d)$ and for every multi-index~$\mathbf{k}$
  \[ \label{eq:varphi_in_H2}
  \sup_{t \in [0,T]} \int_{\R^d} \big(|\D^\mathbf{k}_x\varphi(t,x)|^2 + |\de_{tt} \D^\mathbf{k}_x \varphi(t,x)|^2 \big) \, \d x  < +\infty \, ,
  \]
  Let us prove the first bound in~\eqref{eq:varphi_in_H2}.
  We start by noticing that the Fourier transform of~$\varphi$ is given~by 
   \[
   \hat \varphi(t,\xi) = - \int_t^T \frac{\sin(\omega_\delta(\xi)(t-\tau))}{\omega_\delta(\xi)} \hat \psi(\tau,\xi) \, \d \tau \, .
   \]
   This allows us to estimate, using Plancherel's identity,
   \[ 
    \begin{split}
      & \int_{\R^d} \big| \D^\mathbf{k}_x \varphi(t,x)\big|^2 \, \d x = \int_{\R^d} \big| \hat{\D^\mathbf{k}_x \varphi}(t,\xi)\big|^2 \, \d \xi \leq C \int_{\R^d} |\xi|^{2|\mathbf{k}|} |\hat \varphi(t,\xi)|^2 \, \d \xi \\
      & \quad = C \int_{\R^d} |\xi|^{2|\mathbf{k}|} \Big| \int_t^T \frac{\sin(\omega_\delta(\xi)(t-\tau))}{\omega_\delta(\xi)} \hat \psi(\tau,\xi) \, \d \tau \Big|^2 \, \d \xi \leq  C \int_{\R^d} |\xi|^{2|\mathbf{k}|} \int_t^T |\tau - t|^2 |\hat \psi(\tau,\xi)|^2 \, \d \tau \, \d \xi \\
      & \quad \leq C T^2 \int_0^T \int_{\R^d} |\xi|^{2|\mathbf{k}|} |\hat \psi(\tau,\xi)|^2  \, \d \xi \, \d \tau  \leq C < +\infty \quad \text{for every } t \in [0,T] \, ,
      \end{split}
   \]
   where we used that~$\hat \psi$ is a rapidly decreasing function in the variable~$\xi$. 
   
   To estimate the second term in~\eqref{eq:varphi_in_H2}, we observe that Plancherel's identity and~\eqref{eq:omega_delta_high_frequencies} yield the bounds
   \[ \label{eq:K_less_Lap}
   \begin{split}
   \int_{\R^d} \big|K_\delta^*[f](x) \big|^2 \, \d x & = \int_{\R^d} \omega_\delta^4(\xi) | \hat f(\xi)|^2 \, \d \xi \leq C \int_{\R^d} |\xi|^{4\alpha} | \hat f(\xi)|^2 \, \d \xi \leq C \int_{\R^d} |\xi|^{4} | \hat f(\xi)|^2 \, \d \xi \\
   &\qquad \leq C \int_{\R^d} \big| \D^2_x f(x) \big|^2 \, \d x \, .
   \end{split}
  \]
   In particular,
   \[
   \begin{split}
    & \int_{\R^d} \big|K_\delta^*[\D^\mathbf{k}_x\varphi(t,\cdot)](x)\big|^2 \, \d x \leq C \int_{\R^d} \big|\D^{|\mathbf{k}|+2}_x \varphi(t,x)\big|^2 \, \d x \, , \quad \text{for every } t \in [0,T] \, .
   \end{split}
   \]
   This, together with the fact that~$\D^\mathbf{k}_x \varphi$ is a solution to the problem
   \[
     \begin{cases} 
       \de_{tt} \D^\mathbf{k}_x \varphi(t,x) - K_\delta^*[\D^\mathbf{k}_x \varphi(t,\cdot)](x) = \D^\mathbf{k}_x \psi(t,x) \, , &   (t,x) \in (0,T) \x \R^d \, , \\
       \D^\mathbf{k}_x \varphi(T,x) = 0 \, , \quad \de_t \D^\mathbf{k}_x \varphi(T,x) = 0 \, , &  x \in \R^d \, ,
     \end{cases} 
     \]
    yields the second estimate in~\eqref{eq:varphi_in_H2}.

  \underline{\itshape Correcting $\varphi$}. We correct~$\varphi$ by approximating it with a sequence of test functions~$\varphi_n \in \Dcal(\R \x \R^d;\R^d)$ as follows. We consider a cut-off function~$\zeta \in C^\infty_c([0,+\infty))$ such that~$0 \leq \zeta \leq 1$, $\zeta(\rho) = 1$ for~$\rho \in [0,1]$, and~$\zeta(\rho) = 0$ for~$\rho \geq 2$ and we let 
  \[
  \zeta_n(\rho) = \begin{cases}
    1 \, , & \text{for } \rho \in [0,n] \,, \\ 
    \zeta(\rho-n) \, , & \text{for } \rho \in [n,n+2] \, , \\
    0 \, , & \text{for } \rho \in [n+2,+\infty) \, .
  \end{cases}
  \]
  Finally, we set~$\varphi_n(t,x) := \zeta_n(|t|) \zeta_n(|x|) \varphi(t,x)$.  
  Below, we will always assume that~$n > T$, so that for every~$t \in [0,T]$ and~$x \in \R^d$ we simply have that 
  \[ \label{eq:varphi_n}
    \varphi_n(t,x) = \zeta_n(|x|) \varphi(t,x) \, .
  \]
  Let us compute the derivatives of this modification. 
  For a multi-index~$\mathbf{k}$, we have that 
  \[ \label{eq:Leibniz_rule}
  \begin{split}
  \D^\mathbf{k}_x \varphi_n(t,x) & = \sum_{\mathbf{i} + \mathbf{j} = \mathbf{k}} \frac{\mathbf{k}!}{\mathbf{i}! \mathbf{j}!} \D^\mathbf{i}_x \big(\zeta_n(|\cdot|)\big)(x)  \D^\mathbf{j}_x \varphi(t,x) \\
    & = \sum_{\substack{\mathbf{i} + \mathbf{j} = \mathbf{k} \\ |\mathbf{i}| \neq 0}} \frac{\mathbf{k}!}{\mathbf{i}! \mathbf{j}!} \D^\mathbf{i}_x \big(\zeta_n(|\cdot|)\big)(x)  \D^\mathbf{j}_x \varphi(t,x) + \zeta_n(|x|) \D^\mathbf{k}_x \varphi(t,x)\, .
  \end{split}
  \]
  It follows that 
  \[ \label{eq:varphi_n_leq_varphi}
  \big|\D^\mathbf{k}_x \varphi_n(t,x)\big| \leq C \sum_{|\mathbf{j}| \leq |\mathbf{k}|} \big|\D^\mathbf{j}_x \varphi(t,x)\big|   \quad \text{and} \quad
  \big|\de_{tt}\D^\mathbf{k}_x \varphi_n(t,x)\big| \leq C \sum_{|\mathbf{j}| \leq |\mathbf{k}|} \big|\de_{tt}\D^\mathbf{j}_x \varphi(t,x)\big| \, , 
  \]
  where the constant~$C$ depends on~$\zeta$.

  Equality~\eqref{eq:varphi_n} also implies that 
  \[ \label{eq:varphi_n_equals_varphi}
    |x| < n \implies \varphi_n(t,x) = \varphi(t,x) \, ,
  \]
  and, accordingly, the equality holds for derivatives. 

  \underline{\itshape Estimating the error}. We observe that, for every multi-index~$\mathbf{k}$, 
  \[ \label{eq:error_to_zero}
 \lim_{n\to+\infty} \sup_{t \in [0,T]} \Big( \big\| \de_{tt}( \D^\mathbf{k}_x\varphi(t,\cdot) - \D^\mathbf{k}_x\varphi_n(t,\cdot))\big\|_{L^2}  +  \big\| K_\delta^*[ \D^\mathbf{k}_x\varphi(t,\cdot) - \D^\mathbf{k}_x\varphi_n(t,\cdot)]\big\|_{L^2} \Big) = 0 \,.
  \]
  
  To show the first convergence in~\eqref{eq:error_to_zero}, we use~\eqref{eq:varphi_n_equals_varphi}, \eqref{eq:varphi_n_leq_varphi}, and~\eqref{eq:varphi_in_H2} to deduce that 
  \[
  \begin{split}
  &  \lim_{n\to+\infty}\int_{\R^d}\big| \de_{tt}( \D^\mathbf{k}_x\varphi(t,x) - \D^\mathbf{k}_x\varphi_n(t,x)) \big|^2 \, \d x 
  \\&\qquad
  = \lim_{n\to+\infty} \int_{\R^d \sm \{|x| < n\}}\big| \de_{tt}( \D^\mathbf{k}_x\varphi(t,x) - \D^\mathbf{k}_x\varphi_n(t,x)) \big|^2 \, \d x \\
  & \qquad \leq C \lim_{n\to+\infty} \sum_{|\mathbf{j}| \leq |\mathbf{k}|}\int_{\R^d \sm \{|x| < n\}}\big| \de_{tt}( \D^\mathbf{j}_x\varphi(t,x))\big|^2 \, \d x = 0 \, .
  \end{split}
  \]
  To prove the second convergence in~\eqref{eq:error_to_zero}, we employ~\eqref{eq:K_less_Lap}, \eqref{eq:varphi_n_equals_varphi}, and~\eqref{eq:varphi_in_H2} to conclude that 
  \[
  \begin{split}
    &  \lim_{n\to+\infty}\int_{\R^d} \big|K_\delta^*[\varphi(t,\cdot) - \varphi_n(t,\cdot)](x) \big|^2 \, \d x \leq C \lim_{n\to+\infty} \int_{\R^d} \big| \D^{|\mathbf{k}|+2}_x \varphi(t,x)  - \D^{|\mathbf{k}|+2}_x \varphi_n(t,x) \big|^2 \, \d x  \\
    & \qquad\qquad \leq C \lim_{n\to+\infty} \int_{\R^d \sm \{|x| < n\}} \big| \D^{|\mathbf{k}|+2}_x \varphi(t,x)  - \D^{|\mathbf{k}|+2}_x \varphi_n(t,x) \big|^2 \, \d x  \\
    & \qquad\qquad \leq C \lim_{n\to+\infty} \sum_{|\mathbf{j}| \leq |\mathbf{k}|+2} \int_{\R^d \sm \{|x| < n\}} \big| \D^{\mathbf{j}}_x \varphi(t,x) \big|^2 \, \d x= 0  \, .
  \end{split}
  \]
 
  \underline{\itshape Showing that $\langle u_\delta , \psi \rangle = 0$}.
  We write the duality as follows:
  \[ \label{eq:split_with_approximation}
  \begin{split}
  \langle u_\delta , \psi \rangle & = \langle u_\delta , \de_{tt} \varphi - K_\delta^*[\varphi] \rangle = \langle u_\delta , \de_{tt} \varphi_n - K_\delta^*[\varphi_n] \rangle + \langle u_\delta , \de_{tt} (\varphi - \varphi_n) - K_\delta^*[\varphi - \varphi_n] \rangle \, .
  \end{split}
  \]
  By Definition~\ref{def:distributional_solution}, using the initial condition~$(u_0,v_0)=(0,0)$, we get that
  \[
    \langle u_\delta , \de_{tt} \varphi_n - K_\delta^*[\varphi_n] \rangle = 0 \,, \quad \text{for every } n \in \N \, ,
  \]
  hence we only need to estimate the second term in~\eqref{eq:split_with_approximation}. 
  If~$s \geq 0$, we have the representation
  \[  \label{eq:case_s_positive}
  \begin{split}
  & \big| \langle u_\delta , \de_{tt} (\varphi - \varphi_n) - K_\delta^*[\varphi - \varphi_n] \rangle \big| \\
 = \;&\left| \int_0^T \int_{\R^d} u_\delta(t,\cdot) \cdot \big( \de_{tt} (\varphi(t,\cdot) - \varphi_n(t,\cdot)) - K_\delta^*[\varphi(t,\cdot) - \varphi_n(t,\cdot)] \big)  \, \d t \right|  \\
 \leq\;& \int_0^T \big\|u_\delta(t,\cdot)\big\|_{L^2(\R^d;\R^d)} \Big( \big\|\de_{tt}  (\varphi(t,\cdot) - \varphi_n(t,\cdot))\big\|_{L^2(\R^d;\R^d)} + \big\| K_\delta^*[\varphi(t,\cdot) - \varphi_n(t,\cdot)] \big\|_{L^2(\R^d;\R^d)} \Big) \d t \, .
  \end{split}
  \] 
  Instead, if~$s < 0$, then 
  \[ \label{eq:case_s_negative}
  \begin{split}
  & \big| \langle u_\delta, \de_{tt} (\varphi - \varphi_n) - K_\delta^*[\varphi - \varphi_n] \rangle \big| = \left| \int_0^T \int_{\R^d} \langle u_\delta, \de_{tt} (\varphi - \varphi_n) - K_\delta^*[\varphi - \varphi_n] \rangle_{H^{-|s|}, H^{|s|}} \, \d t \right|  \\
  &  \quad \leq \int_0^T \big\|u_\delta(t,\cdot)\big\|_{H^{-|s|}(\R^d;\R^d)} \Big( \big\|\de_{tt}  (\varphi(t,\cdot) - \varphi_n(t,\cdot))\big\|_{H^{|s|}(\R^d;\R^d)} \\
  & \hphantom{\quad \leq \int_0^T \big\|u_\delta(t,\cdot)\big\|_{H^{-|s|}(\R^d;\R^d)} \Big( }  \quad \quad + \big\| K_\delta^*[\varphi(t,\cdot) - \varphi_n(t,\cdot)] \big\|_{H^{|s|}(\R^d;\R^d)} \Big) \d t \\
  & \quad \leq \int_0^T \big\|u_\delta(t,\cdot)\big\|_{H^{-|s|}(\R^d;\R^d)} \Big( \sum_{|\mathbf{k}| \leq m}   \big\|\de_{tt}  (\D_x^\mathbf{k}\varphi(t,\cdot) - \D_x^\mathbf{k}\varphi_n(t,\cdot))\big\|_{L^2(\R^d;\R^d)}  \\
  & \hphantom{\quad \leq \int_0^T \big\|u_\delta(t,\cdot)\big\|_{H^{-|s|}(\R^d;\R^d)}  } \quad + \sum_{|\mathbf{k}| \leq m}  \big\| K_\delta^*[\D_x^\mathbf{k}\varphi(t,\cdot) - \D_x^\mathbf{k}\varphi_n(t,\cdot)] \big\|_{L^2(\R^d;\R^d)}  \Big) \, \d t
  \end{split}
  \]
  In both cases, from~\eqref{eq:error_to_zero} we deduce the convergence to zero of the integrands in the right-hand sides of~\eqref{eq:case_s_positive} and~\eqref{eq:case_s_negative}. 
  We conclude by observing that~\eqref{eq:varphi_n_leq_varphi}, \eqref{eq:varphi_in_H2}, and the assumption $u \in L^1_\mathrm{loc}((0,+\infty);H^s(\R^d;\R^d))$ allow us to apply the Dominated Convergence Theorem and conclude that 
  \[
    \lim_{n\to+\infty} \big| \langle u_\delta, \de_{tt} (\varphi - \varphi_n) - K_\delta^*[\varphi - \varphi_n] \rangle \big| = 0 \, .
  \]
  This concludes the proof.
\end{proof}

\subsection{Energy conservation} The total energy~$E^{K_\delta}$ is defined in~\eqref{eq:total_energy}.
We recall the following result (see, \eg, \cite{CocDipFanMadVal22}).

\begin{theorem}
  Let~$u_0 \in H^\alpha(\R^d;\R^d)$ and~$v_0 \in L^2(\R^d;\R^d)$.
  Let $u_\delta \in C([0,+\infty);H^\alpha(\R^d;\R^d)) \cap C^1((0,+\infty);L^2(\R^d;\R^d))$ be the unique solution to the linear peridynamics model in~\eqref{eq:linear_peridynamics} with initial data~$(u_0,v_0)$ provided by Theorems~\ref{thm:existence} and~\ref{thm:uniqueness}.

Then, the total energy~$E^{K_\delta}[u_\delta](t)$ is conserved in time, \ie, $E^{K_\delta}[u_\delta](t) = E^{K_\delta}[u_\delta](0)$ for every~$t \in [0,+\infty)$.
\end{theorem}

\begin{proof}
  We provide the proof for completeness. 
  We use Plancherel's identity, Lemma~\ref{lem:W_in_Fourier_space} and formula~\eqref{eq:Fourier_solution} to write the energy as
  \[
  \begin{split}
  E^{K_\delta}[u_\delta](t) & = \frac{1}{2} \int_{\R^d} |\de_t u_\delta(t,x)|^2 \, \d x + \frac{1}{2} [u_\delta(t,\cdot)]_{\W_\delta^\alpha}^2 \\
  & = \frac{1}{2} \int_{\R^d} |\de_t \hat u_\delta(t,\xi)|^2 \, \d \xi + \frac{1}{2}\int_{\R^d} \omega_\delta^2(\xi) |\hat u_\delta(t,\xi)|^2 \d \xi \\
  & = \frac{1}{2} \int_{\R^d} |-\omega_\delta(\xi) \sin(\omega_\delta(\xi)t) \hat u_0(\xi) + \cos(\omega_\delta(\xi)t) \hat v_0(\xi)|^2 \, \d \xi \\
  &\qquad\quad + \frac{1}{2}\int_{\R^d} \omega_\delta^2(\xi) \Big|\cos(\omega_\delta(\xi)t) \hat u_0(\xi) + \frac{\sin(\omega_\delta(\xi)t)}{\omega_\delta(\xi)} \hat v_0(\xi)\Big|^2 \d \xi \\
  & = \frac{1}{2} \int_{\R^d} |\hat v_0(\xi)|^2 \, \d \xi + \frac{1}{2} \int_{\R^d} \omega_\delta^2(\xi)|\hat u_0(\xi)|^2 \, \d \xi \\
  & = \frac{1}{2} \int_{\R^d} |v_0(x)|^2 \, \d x + \frac{1}{2} [u_0]_{\W_\delta^\alpha}^2\\& = E^{K_\delta}[u_\delta](0) \, .
  \end{split}
  \]
  This concludes the proof.
\end{proof}
 
\section{Comparison with solutions to the wave equation: Small $\delta$} 
\label{deltasmall76859403}

Before delving into the main results of the paper, we recall some basic facts about the classical wave equation in~\eqref{eq:wave}.

We recall the notion of distributional solution. 
\begin{definition} \label{def:distributional_solution_wave}
  Let~$u_0 \in \Dcal'(\R^d;\R^d)$ and~$v_0 \in \Dcal'(\R^d;\R^d)$. A distributional solution to the wave equation in~\eqref{eq:wave} with initial data~$(u_0,v_0)$ is a distribution $u \in \Dcal'(\R \x \R^d;\R^d)$ such that $\supp(u) \subset [0,+\infty) \x \R^d$ and
  \[  
  \langle u, \de_{tt} \varphi - \gamma^2\Delta \varphi \rangle =  - \langle u_0, \de_t \varphi(0,\cdot) \rangle + \langle v_0, \varphi(0,\cdot) \rangle \, ,    
  \]
  for every~$\varphi \in \Dcal(\R \x \R^d;\R^d)$.
  
  In the above formula, the duality pairing in the left-hand side is between $\Dcal'(\R \x \R^d;\R^d)$ and $\Dcal(\R \x \R^d;\R^d)$, while the duality pairings in the right-hand side are between $\Dcal'(\R^d;\R^d)$ and $\Dcal(\R^d;\R^d)$.
\end{definition}

The next result is well-known. 
\begin{theorem} \label{thm:existence_wave}
  Let~$s \in \R$. 
  Let~$u_0 \in H^s(\R^d;\R^d)$ and~$v_0 \in H^{s-1}(\R^d;\R^d)$. 
  
  Then, there exists $u \in C([0,+\infty);H^s(\R^d;\R^d)) \cap C^1((0,+\infty);H^{s-1}(\R^d;\R^d))$ such that, for every~$t \in [0,+\infty)$, 
  \[ \label{eq:Fourier_wave}
    \hat{u} (t,\cdot) = \F[u(t,\cdot)]  = \cos(\gamma |\cdot| t) \hat{u}_0  + \frac{\sin(\gamma |\cdot| t)}{\gamma|\cdot|} \hat{v}_0 \, .
  \]
  Moreover, $u$ is the unique distributional solution in $\Dcal'(\R \x \R^d;\R^d)$ to the wave equation in~\eqref{eq:wave} with initial data~$(u_0,v_0)$.
\end{theorem} 

\begin{remark} \label{rem:more_regular_v0}
  A simple remark is needed for the statement of Theorem~\ref{thm:delta_to_0}.
  Let~$u_0 \in H^s(\R^d;\R^d)$ and~$v_0 \in H^{s-\alpha}(\R^d;\R^d)$. 
  Then we can use the fact that~$v_0 \in H^{s-1}(\R^d;\R^d)$ to find a solution~$u \in C([0,+\infty);H^s(\R^d;\R^d)) \cap C^1((0,+\infty);H^{s-1}(\R^d;\R^d))$ to the wave equation in~\eqref{eq:wave} with initial data~$(u_0,v_0)$.
\end{remark}

We are now in a position to show that solutions to the linear peridynamics model approximate solutions to the wave equation as~$\delta \to 0$. 

\begin{proof}[Proof of Theorem~\ref{thm:delta_to_0}]
  We prove the two convergences separately.

  \underline{\itshape Estimating $u_\delta - u$}.
  Let us start by estimating $\| u_\delta(t,\cdot) - u(t,\cdot) \|_{H^s(\R^d;\R^d)}$. 
  By~\eqref{eq:Fourier_solution} and~\eqref{eq:Fourier_wave}, for every~$t \in [0,T]$ we have that 
  \[ \label{eq:Fourier_difference}
  \begin{split}
  & \|u_\delta(t,\cdot) - u(t,\cdot) \|_{H^s(\R^d;\R^d)} \leq C \left( \int_{\R^d} (1+|\xi|^2)^s |\hat u_\delta(t,\xi) - \hat u(t,\xi)|^2 \d \xi \right)^\frac{1}{2} \\
  & \qquad\quad \leq  C \left( \int_{\R^d} (1+|\xi|^2)^s |\hat u_0(\xi)|^2 |\cos(\omega_\delta(\xi)t) - \cos(\gamma |\xi| t)|^2 \d \xi \right)^\frac{1}{2} \\
  & \qquad\qquad\quad \quad + C \left( \int_{\R^d} (1+|\xi|^2)^s |\hat v_0(\xi)|^2 \left|\frac{\sin(\omega_\delta(\xi)t)}{\omega_\delta(\xi)} - \frac{\sin(\gamma |\xi| t)}{\gamma |\xi|}\right|^2 \d \xi \right)^\frac{1}{2}  \,.
  \end{split}
  \]
  
  Let us consider the first integral in the right-hand side of~\eqref{eq:Fourier_difference}. 
  By Lemma~\ref{lem:difference_omega_delta_gamma_xi}, we have that,
for every~$\xi \in \R^d$,
\[
  \lim_{\delta\to0}  \sup_{t \in [0,T]} \big|\cos(\omega_\delta(\xi)t) - \cos(\gamma |\xi| t)\big| \leq \lim_{\delta\to0}\sup_{t \in [0,T]} \big|\omega_\delta(\xi)t - \gamma |\xi| t\big| \leq \lim_{\delta\to0}C \delta |\xi|^2 T = 0  \, .
  \]
  Since $(1+|\xi|^2)^\frac{s}{2} \hat u_0(\xi) \in L^2(\R^d;\C^d)$, by the Dominated Convergence Theorem we deduce that 
  \[ \label{eq:limit of cos part}
 \lim_{\delta\to0} \sup_{t \in [0,T]} \left( \int_{\R^d} (1+|\xi|^2)^s |\hat u_0(\xi)|^2 |\cos(\omega_\delta(\xi)t) - \cos(\gamma |\xi| t)|^2 \d \xi \right)^\frac{1}{2}= 0 \, .
  \]
  
  Let us consider the second integral in the right-hand side of~\eqref{eq:Fourier_difference}. 
  We split it as follows
  \begin{equation} \label{eq:split integral}
    \begin{split}
      & \int_{\R^d} (1+|\xi|^2)^s | \hat v_0(\xi)|^2 \Big|\frac{\sin(\omega_\delta(\xi) t)}{\omega_\delta(\xi)} - \frac{\sin(\gamma |\xi| t)}{\gamma |\xi|} \Big|^2  \,  \d \xi \\
      &  =  \int_{ \{|\xi| \leq 1\} } (1+|\xi|^2)^s | \hat v_0(\xi)|^2 \Big|\frac{\sin(\omega_\delta(\xi) t)}{\omega_\delta(\xi)} - \frac{\sin(\gamma |\xi| t)}{\gamma |\xi|} \Big|^2  \,  \d \xi  \\
      & \qquad\quad + \int_{ \{|\xi| > 1\} } (1+|\xi|^2)^s | \hat v_0(\xi)|^2 \Big|\frac{\sin(\omega_\delta(\xi) t)}{\omega_\delta(\xi)} - \frac{\sin(\gamma |\xi| t)}{\gamma |\xi|} \Big|^2  \,  \d \xi \, .
    \end{split}
  \end{equation}
  By the Lipschitz continuity of~$\frac{\sin(z)}{z}$ and by Lemma~\ref{lem:difference_omega_delta_gamma_xi}, we determine the limit of the first part as follows 
  \begin{equation} \label{eq:int xi leq 1}
    \begin{split}
      &\lim_{\delta\to0} \sup_{t \in [0,T]}\int_{ \{|\xi| \leq 1\} } (1+|\xi|^2)^s | \hat v_0(\xi)|^2 \Big|\frac{\sin(\omega_\delta(\xi) t)}{\omega_\delta(\xi)} - \frac{\sin(\gamma |\xi| t)}{\gamma |\xi|} \Big|^2  \,  \d \xi \\
      & \quad  \leq \lim_{\delta\to0} \sup_{t \in [0,T]} \int_{ \{|\xi| \leq 1\} } 2^s | \hat v_0(\xi)|^2 t^2  \big| \omega_\delta(\xi)t  - \gamma |\xi| t \big|^2  \,  \d \xi
      \\& \quad \leq \lim_{\delta\to0}  \int_{ \{|\xi| \leq 1\} } 2^s | \hat v_0(\xi)|^2 T^4 C \delta^2 |\xi|^4  \,  \d \xi \\
      & \quad \leq \lim_{\delta\to0} C T^4  \delta^2 \int_{ \{|\xi| \leq 1\} } | \hat v_0(\xi)|^2   \,  \d \xi= 0 \, .
    \end{split}
  \end{equation}
  For the second part in~\eqref{eq:split integral}, we observe that Lemma~\ref{lem:difference_omega_delta_gamma_xi} and the Lipschitz continuity of~$\frac{\sin(z)}{z}$ yield that,
 for every~$\xi \in \R^d$,
  \[
   \lim_{\delta\to0} \sup_{t \in [0,T]}\Big|\frac{\sin(\omega_\delta(\xi) t)}{\omega_\delta(\xi)} - \frac{\sin(\gamma |\xi| t)}{\gamma |\xi|} \Big| \leq
   \lim_{\delta\to0} \sup_{t \in [0,T]} \big|\omega_\delta(\xi)t^2 - \gamma |\xi| t^2\big| \leq C T^2 \delta |\xi|^2 = 0 \, .
    \]

    We use Lemma~\ref{lem:omega_delta_low_high_frequencies} to exploit the bounds $\sup_{|\xi| > 1} \frac{(1+|\xi|^2)^\alpha}{\omega_\delta^2(\xi)} < +\infty$ and $\sup_{|\xi| > 1} \frac{\omega_\delta^2(\xi)}{|\xi|} < +\infty$. 
    These allow us to estimate for~$|\xi| > 1$
  \[
    \begin{split}
      &  (1+|\xi|^2)^s | \hat v_0(\xi)|^2  \sup_{t \in [0,T]} \Big|\frac{\sin(\omega_\delta(\xi) t)}{\omega_\delta(\xi)} - \frac{\sin(\gamma |\xi| t)}{\gamma |\xi|} \Big|^2   \\
      & \quad =   \frac{(1+|\xi|^2)^{\alpha}}{\omega_\delta^2(\xi)} (1+|\xi|^2)^{s - \alpha}  | \hat v_0(\xi)|^2 \sup_{t \in [0,T]} \Big| \sin(\omega_\delta(\xi) t)  - \sin(\gamma |\xi| t) \frac{\omega_\delta(\xi)}{\gamma |\xi|} \Big|^2   \\
      & \quad \leq \left(\sup_{|\xi|> 1}\frac{(1+|\xi|^2)^{\alpha}}{\omega_\delta^2(\xi)} \right) (1+|\xi|^2)^{s - \alpha}  
      | \hat v_0(\xi)|^2 \left( 2 + 2\sup_{|\xi| > 1} \frac{\omega_\delta^2(\xi)}{ \gamma^2 |\xi|^2}  \right)     \, .
    \end{split}
  \]
  Recalling that integrability of $(1+|\xi|^2)^{\frac{s - \alpha}{2}}  \hat v_0(\xi) \in L^2(\R^d;\C^d)$, by the Dominated Convergence Theorem we conclude that 
  \begin{equation} \label{eq:int xi geq 1}
   \lim_{\delta\to0} \sup_{t \in [0,T]}  \int_{ \{|\xi| > 1\} } (1+|\xi|^2)^s | \hat v_0(\xi)|^2  \Big|\frac{\sin(\omega_\delta(\xi) t)}{\omega_\delta(\xi)} - \frac{\sin(\gamma |\xi| t)}{\gamma |\xi|} \Big|^2 \,  \d \xi = 0 \, .
  \end{equation}
  Combining~\eqref{eq:int xi leq 1} and~\eqref{eq:int xi geq 1} in~\eqref{eq:split integral} gives 
  \begin{equation} \label{eq:limit of sin part}
  \lim_{\delta\to0}  \sup_{t \in [0,T]} \int_{ \R^d } (1+|\xi|^2)^s  | \hat v_0(\xi)|^2 \Big|\frac{\sin(\omega_\delta(\xi) t)}{\omega_\delta(\xi)} - \frac{\sin(\gamma |\xi| t)}{\gamma |\xi|} \Big|^2  \,  \d \xi = 0 \, .
  \end{equation}

  Summing~\eqref{eq:limit of cos part} and~\eqref{eq:limit of sin part}, by~\eqref{eq:Fourier_difference} we conclude that 
  \[
  \lim_{\delta\to0}
  \sup_{t \in [0,T]}   \|  u_\delta(t,\cdot) -   u(t,\cdot) \|_{H^s(\R^d;\R^d)}  = 0 \,  .
  \]

  \underline{\itshape Estimating $\de_t u_\delta - \de_t u$}. By~\eqref{eq:Fourier_solution} and~\eqref{eq:Fourier_wave}, for every~$t \in [0,T]$ we have that 
  \[
  \begin{split}  
  \de_t \hat u_\delta(t,\cdot) & = -\omega_\delta(\cdot) \sin(\omega_\delta(\cdot)t) \hat u_0 + \cos(\omega_\delta(\cdot)t) \hat v_0  \\
  {\mbox{and }}\qquad
  \de_t \hat u(t,\cdot) & = -\gamma|\cdot| \sin(\gamma|\cdot|t) \hat u_0 + \cos(|\cdot|t) \hat v_0 \, .
\end{split}
  \]
  It follows that 
  \[  \label{eq:Fourier_difference_derivative}
    \begin{split}
    & \|\de_t u_\delta(t,\cdot) - \de_t u(t,\cdot) \|_{H^{s-1}(\R^d;\R^d)} \leq C \left( \int_{\R^d} (1+|\xi|^2)^{s-1} \big|\de_t \hat u_\delta(t,\xi) - \de_t \hat u(t,\xi)\big|^2 \d \xi \right)^\frac{1}{2} \\
    & \qquad \leq  C \left( \int_{\R^d} (1+|\xi|^2)^{s-1} |\hat u_0(\xi)|^2 \big|\omega_\delta(\xi)\sin(\omega_\delta(\xi)t) - \gamma |\xi| \sin(\gamma |\xi| t)\big|^2 \d \xi \right)^\frac{1}{2} \\
    & \qquad\quad \quad + C \left( \int_{\R^d} (1+|\xi|^2)^{s-1} |\hat v_0(\xi)|^2 \big|\cos(\omega_\delta(\xi)t) - \cos(\gamma |\xi| t) \big|^2 \d \xi \right)^\frac{1}{2}  \,.
    \end{split}
  \]
  To study the first integral in~\eqref{eq:Fourier_difference_derivative}, we note that Lemma~\ref{lem:difference_omega_delta_gamma_xi} yields that,
for every~$\xi \in \R^d$,
  \[
  \begin{split}
  & \lim_{\delta\to0}\sup_{t \in [0,T]}\big|\omega_\delta(\xi)\sin(\omega_\delta(\xi)t) - \gamma |\xi| \sin(\gamma |\xi| t)\big| \\
  &\quad \leq \lim_{\delta\to0}\sup_{t \in [0,T]} \Big( \big|\omega_\delta(\xi)  - \gamma |\xi|   \big| +  \gamma |\xi| \big| \sin(\omega_\delta(\xi)t) -   \sin(\gamma |\xi| t)\big| \Big)\\
  & \quad \leq \lim_{\delta\to0}(1 + \gamma|\xi| T)  \big|\omega_\delta(\xi)  - \gamma |\xi|   \big|  \leq \lim_{\delta\to0}C (1 + \gamma|\xi| T) \delta |\xi|^2 = 0  \, .
  \end{split}
  \]
  Moreover, by Lemma~\ref{lem:omega_delta_low_high_frequencies} we have that~$\omega_\delta(\xi) \leq C|\xi|$, hence  
  \[
    (1+|\xi|^2)^{s-1} |\hat u_0(\xi)|^2 \big|\omega_\delta(\xi)\sin(\omega_\delta(\xi)t) - \gamma |\xi| \sin(\gamma |\xi| t)\big|^2 \leq C (1+|\xi|^2)^{s} |\hat u_0(\xi)|^2 \in L^1(\R^d) \, .
  \]
  We apply the Dominated Convergence Theorem to conclude that 
  \[ \label{eq:limit of sin part derivative}
  \lim_{\delta\to0}\sup_{t \in [0,T]} \left( \int_{\R^d} (1+|\xi|^2)^{s-1} |\hat u_0(\xi)|^2 \big|\omega_\delta(\xi)\sin(\omega_\delta(\xi)t) - \gamma |\xi| \sin(\gamma |\xi| t)\big|^2 \d \xi \right)^\frac{1}{2} = 0 \, .
  \]
  
  To study the second integral in~\eqref{eq:Fourier_difference_derivative}, we observe that Lemma~\ref{lem:difference_omega_delta_gamma_xi} yields
  that, for every~$\xi \in \R^d$, 
  \[
\lim_{\delta\to0}  \sup_{t \in [0,T]}(1+|\xi|^2)^{s-1} |\hat v_0(\xi)|^2 \big|\cos(\omega_\delta(\xi)t) - \cos(\gamma |\xi| t) \big|^2 \leq C (1+|\xi|^2)^{s-1} |\hat v_0(\xi)|^2 T^2 \delta^2 |\xi|^4 = 0 \, .
  \]
  Moreover, $v_0 \in H^{s-\alpha}(\R^d;\R^d) \subset H^{s-1}(\R^d;\R^d)$, hence
  \[
    (1+|\xi|^2)^{s-1} |\hat v_0(\xi)|^2 \big|\cos(\omega_\delta(\xi)t) - \cos(\gamma |\xi| t) \big|^2 \leq 4 (1+|\xi|^2)^{s-1} |\hat v_0(\xi)|^2  \in L^1(\R^d) \, .
  \]
  We can apply the Dominated Convergence Theorem to conclude that
  \[ \label{eq:limit of cos part derivative}
   \lim_{\delta\to0}\sup_{t \in [0,T]} \left( \int_{\R^d} (1+|\xi|^2)^{s-1} |\hat v_0(\xi)|^2 \big|\cos(\omega_\delta(\xi)t) - \cos(\gamma |\xi| t) \big|^2 \d \xi \right)^\frac{1}{2} =0 \, .
  \]

  Summing~\eqref{eq:limit of sin part derivative} and~\eqref{eq:limit of cos part derivative}, by~\eqref{eq:Fourier_difference_derivative} we deduce that
  \[
  \lim_{\delta\to0} \sup_{t \in [0,T]}   \|  \de_t u_\delta(t,\cdot) -   \de_t u(t,\cdot) \|_{H^{s-1}(\R^d;\R^d)}  = 0\,  .
  \]
  This concludes the proof.
\end{proof}

\begin{remark} \label{rem:initial_data}
  We stress the importance of taking initial data in the correct spaces in Theorem~\ref{thm:delta_to_0}.

  Specifically, the space for the initial displacement and the space for the initial velocity must exhibit a regularity gap given by~$\alpha \in (0,1)$, consistently with the regularity of the solutions to the linear peridynamics model.
  For instance, if~$u_0 \in H^1(\R^d;\R^d)$ and~$v_0 \in H^{1-\alpha}(\R^d;\R^d)$, then 
  \[ \lim_{\delta\to0}
  \sup_{t \in [0,T]} \Big( \| u_\delta(t,\cdot) - u(t,\cdot) \|_{H^1(\R^d;\R^d)} + \| \de_t u_\delta(t,\cdot) - \de_t u(t,\cdot) \|_{L^2(\R^d;\R^d)} \Big) = 0  \, ,
  \]
  \ie, the norms used in the convergence are consistent with the energy space of the classical wave equation.

  If, instead, initial data are chosen consistently with the wave equation, \ie, $u_0 \in H^1(\R^d;\R^d)$ and~$v_0 \in L^2(\R^d;\R^d)$ (with a regularity gap equal to~$1 > \alpha$), then one gets convergence of solutions to the linear peridynamics model to solutions to the wave equation as follows. 
  We use the fact that~$u_0 \in H^1(\R^d;\R^d) \subset H^\alpha(\R^d;\R^d)$ to get a regularity gap of~$\alpha$ between~$u_0$ and~$v_0$. Then the solution to the linear peridynamics model belongs to~$C([0,+\infty);H^\alpha(\R^d;\R^d)) \cap C^1((0,+\infty);L^2(\R^d;\R^d))$ and Theorem~\ref{thm:delta_to_0} gives the convergence 
  \[ 
\lim_{\delta\to0}
  \sup_{t \in [0,T]} \Big( \| u_\delta(t,\cdot) - u(t,\cdot) \|_{H^\alpha(\R^d;\R^d)} + \| \de_t u_\delta(t,\cdot) - \de_t u(t,\cdot) \|_{H^{\alpha-1}(\R^d;\R^d)} \Big) = 0 \, ,
  \]
  \ie, the norms used in the convergence are weaker than those consistent with the energy space of the classical wave equation. 
 \emph{A posteriori}, $u \in C([0,+\infty);H^1(\R^d;\R^d)) \cap C^1((0,+\infty);L^2(\R^d;\R^d))$ by uniqueness of the solution to the wave equation. 
\end{remark}

\section{Comparison with solutions to the wave equation: Low frequencies}
\label{5647382prooft362y584}

In this section we study quantitatively the distance between solutions to the linear peridynamics model and solutions to the classical wave equation when the initial data are built upon low frequencies. 

\begin{remark} \label{rem:low_frequencies}
  Before stating the result, we observe that if~$u_0$ and~$v_0$ be such that $\supp(\hat u_0) \subset \{|\xi| \leq R\}$ and $\supp(\hat v_0) \subset \{|\xi| \leq R\}$, then the solutions to the linear peridynamics model~$u_\delta$ and to the wave equation~$u$ also satisfy the same regularity, \ie, $\supp(\hat u_\delta(t,\cdot)) \subset \{|\xi| \leq R\}$ and $\supp(\hat u(t,\cdot)) \subset \{|\xi| \leq R\}$ for every~$t \in [0,+\infty)$. 
  This follows from the fact that the Fourier transform of the solution to the linear peridynamics model is given by~\eqref{eq:Fourier_solution} and the Fourier transform of the solution to the wave equation is given by~\eqref{eq:Fourier_wave}.
\end{remark}

\begin{proof}[Proof of Theorem~\ref{thm:low_frequencies}]
  By Remark~\ref{rem:low_frequencies}, formulas~\eqref{eq:Fourier_solution} and~\eqref{eq:Fourier_wave}, and Lemma~\ref{lem:difference_omega_delta_gamma_xi}, for every~$t \in [0,T]$ we have that
  \[
  \begin{split}
  & \| u_\delta(t,\cdot) - u(t,\cdot) \|^2_{H^s(\R^d;\R^d)} = \int_{\{|\xi| \leq R\}} (1+|\xi|^2)^s |\hat u_\delta(t,\xi) - \hat u(t,\xi)|^2 \d \xi \\
  & \quad = \int_{\{|\xi| \leq R\}} (1+|\xi|^2)^s |\hat u_0(\xi)|^2 \Big| \cos(\omega_\delta(\xi)t) - \cos(\gamma |\xi| t) \Big|^2 \d \xi \\
  & \qquad\quad \quad + \int_{\{|\xi| \leq R\}} (1+|\xi|^2)^s |\hat v_0(\xi)|^2 \Big| \frac{\sin(\omega_\delta(\xi)t)}{\omega_\delta(\xi)} - \frac{\sin(\gamma |\xi| t)}{\gamma |\xi|} \Big|^2 \d \xi\\
  & \quad \leq  \int_{\{|\xi| \leq R\}} (1+|\xi|^2)^s \big( |\hat u_0(\xi)|^2 + |\hat v_0(\xi)|^2  \big) \big| \omega_\delta(\xi)t - \gamma |\xi| t \big|^2 \d \xi \\
  & \quad \leq  \int_{\{|\xi| \leq R\}} (1+|\xi|^2)^s \big( |\hat u_0(\xi)|^2 + |\hat v_0(\xi)|^2  \big) C T^2 \delta^2 |\xi|^4 \d \xi \\
  & \quad \leq C T^2 \delta^2 R^4 (1+R^2)^s \int_{\{|\xi| \leq R\}}  \big( |\hat u_0(\xi)|^2 + |\hat v_0(\xi)|^2  \big) \d \xi \\
  & \quad \leq C T^2 \delta^2 R^4 \big( \|u_0\|^2_{L^2(\R^d;\R^d)} + \|v_0\|^2_{L^2(\R^d;\R^d)} \big) \, ,
  \end{split}
  \]
  where we used the assumption~$R \leq 1$ in the last inequality.
 
  Analogously, using also~\eqref{eq:omega_delta_low_frequencies}
  \[
  \begin{split}
  & \| \de_t u_\delta(t,\cdot) - \de_t u(t,\cdot) \|^2_{H^{s-1}(\R^d;\R^d)} = \int_{\{|\xi| \leq R\}} (1+|\xi|^2)^{s-1} \big|\de_t \hat u_\delta(t,\xi) - \de_t \hat u(t,\xi)\big|^2 \d \xi \\
  & \quad = \int_{\{|\xi| \leq R\}} (1+|\xi|^2)^{s-1} |\hat u_0(\xi)|^2 \big| \omega_\delta(\xi)\sin(\omega_\delta(\xi)t) - \gamma |\xi| \sin(\gamma |\xi| t) \big|^2 \d \xi \\
  & \qquad\quad \quad + \int_{\{|\xi| \leq R\}} (1+|\xi|^2)^{s-1} |\hat v_0(\xi)|^2 \big| \cos(\omega_\delta(\xi)t) - \cos(\gamma |\xi| t) \big|^2 \d \xi\\
  & \quad \leq  \int_{\{|\xi| \leq R\}} (1+|\xi|^2)^{s-1}  |\hat u_0(\xi)|^2  \Big( \big| \omega_\delta(\xi)t - \gamma |\xi| t \big| + C |\xi| \big|\sin(\omega_\delta(\xi)t) - \sin(\gamma |\xi| t)  \big| \Big)^2 \d \xi \\
  &\qquad \quad \quad + \int_{\{|\xi| \leq R\}} (1+|\xi|^2)^{s-1}  |\hat v_0(\xi)|^2  \big| \omega_\delta(\xi)t - \gamma |\xi| t \big|^2 \d \xi \\
  & \quad \leq  \int_{\{|\xi| \leq R\}} (1+|\xi|^2)^{s}  |\hat u_0(\xi)|^2   C T^2 \delta^2 |\xi|^4   + \int_{\{|\xi| \leq R\}} (1+|\xi|^2)^{s-1}  |\hat v_0(\xi)|^2  C T^2 \delta^2 |\xi|^4 \d \xi \\
  & \quad \leq C T^2 \delta^2 R^4 (1+R^2)^{s}  \int_{\{|\xi| \leq R\}}  |\hat u_0(\xi)|^2 +  C T^2 \delta^2 R^4 (1+R^2)^{s-1} \int_{\{|\xi| \leq R\}}   |\hat v_0(\xi)|^2  \d \xi \\
  & \quad \leq  C T^2 \delta^2 R^4 \big( \|u_0\|^2_{L^2(\R^d;\R^d)} + \|v_0\|^2_{L^2(\R^d;\R^d)} \big) \,.
  \end{split}
  \]
  This concludes the proof.
\end{proof}

\begin{remark}
  A consequence of Theorem~\ref{thm:low_frequencies} is the following.
  Let~$\e > 0$, $s \geq 0$, $T \in [0,+\infty)$, and~$M > 0$.
  Then there exists a support size~$R_0 > 0$ such that for every~$R \leq R_0$ we have the following.
  
  Given initial data~$u_0$ and~$v_0$ with $\|u_0\|^2_{L^2(\R^d;\R^d)} + \|v_0\|^2_{L^2(\R^d;\R^d)} \leq M$ and $\supp(\hat u_0) \cup \supp(\hat v_0) \subset \{|\xi| \leq R\}$, let~$u_\delta$ be the unique distributional solution to the linear peridynamics model~\eqref{eq:linear_peridynamics} with initial data~$(u_0,v_0)$ provided by Theorems~\ref{thm:existence} and~\ref{thm:uniqueness} and let~$u$ be the unique distributional solution to the wave equation~\eqref{eq:wave} with initial data~$(u_0,v_0)$ provided by Theorem~\ref{thm:existence_wave}. 
  Then,
  \[
  \sup_{t \in [0,T]}  \Big( \| u_\delta(t,\cdot) - u(t,\cdot) \|_{H^s(\R^d;\R^d)} + \| \de_t u_\delta(t,\cdot) - \de_t u(t,\cdot) \|_{H^{s-1}(\R^d;\R^d)}\Big) \leq \e  \, .
  \]
\end{remark}

\begin{remark}
  Since the support of the Fourier transform of the initial data is contained in a given compact set, in Theorem~\ref{thm:low_frequencies} in the right-hand side of the inequalities  one can replace the~$L^2$ norm of the initial data with the~$H^{s'}$ norm of the initial data for any~$s'$. 
\end{remark}

 We conclude this section with the proof of Proposition~\ref{prop:almost energy conservation}.

\begin{proof}[Proof of Proposition~\ref{prop:almost energy conservation}]
The energy conserved by~$u_\delta$ is given by
\[
\begin{split}
E^{K_\delta}[u_\delta](t)  
& = \frac{1}{2}\int_{\R^d} |\de_t u_\delta(t,x)|^2  + \frac{1}{2}[u_\delta(t,\cdot)]_{\W_\delta^\alpha}^2 \, .
\end{split}
\]
We rewrite energy conservation for~$u_\delta$ as follows
\[
\begin{split}&
E^{\gamma^2\Delta}[u_\delta](t)  = E^{K_\delta}[u_\delta](t) + E^{\gamma^2\Delta}[u_\delta](t) - E^{\gamma^2\Delta}[u_\delta](t)\\&\qquad = E^{K_\delta}[u_\delta](0) + E^{\gamma^2\Delta}[u_\delta](t) - E^{K_\delta}[u_\delta](t) \\
& \qquad = E^{\gamma^2\Delta}[u_\delta](0) - E^{\gamma^2\Delta}[u_\delta](0) + E^{K_\delta}[u_\delta](0) + E^{\gamma^2\Delta}[u_\delta](t) - E^{K_\delta}[u_\delta](t) \, ,
\end{split}
\]
\ie,
\[ \label{eq:energy_difference}
  E^{\gamma^2\Delta}[u_\delta](t) = E^{\gamma^2\Delta}[u_\delta](0) - \left(
  \frac{1}{2}[u_\delta(t,\cdot)]_{\W_\delta^\alpha}^2 - \frac{\gamma^2}{2} \|\nabla u_\delta(t,\cdot)\|_{L^2}^2 \right) + \left( \frac{1}{2}[u_0]_{\W_\delta^\alpha}^2 - \frac{\gamma^2}{2} \|\nabla u_0\|_{L^2}^2\right) \, .
\] 

By Lemma~\ref{lem:difference_omega_delta_gamma_xi}, we have that 
\[
\begin{split}
&\Big| [u_0]_{\W_\delta^\alpha}^2 - \gamma^2 \|\nabla u_0\|_{L^2}^2 \Big| = \left| \int_{\R^2} \omega_\delta^2(\xi) |\hat u_0(\xi)|^2 \, \d \xi - \gamma^2 \int_{\R^2} |\xi|^2 |\hat u_0(\xi)|^2 \, \d \xi \right| \\
& \qquad\leq \int_{\R^d} \big| \omega_\delta^2(\xi) - \gamma^2 |\xi|^2 \big| |\hat u_0(\xi)|^2 \, \d \xi  \leq \int_{\R^d} C \delta^2 |\xi|^4 \big| |\hat u_0(\xi)|^2 \, \d \xi  \\
&\qquad \leq C \delta^2 R^2 \int_{\{|\xi| \leq R\}} |\xi|^2 |\hat u_0(\xi)|^2 \, \d \xi = C \delta^2 R^2 \|\nabla u_0\|_{L^2}^2 \, .
\end{split}
\]
Analogously, applying also~\eqref{eq:Fourier_solution},
\[ 
\begin{split}
& \Big| [u_\delta(t,\cdot)]_{\W_\delta^\alpha}^2 - \gamma^2 \|\nabla u_\delta(t,\cdot)\|_{L^2}^2 \Big| \leq C \delta^2 R^2 \int_{\R^d} |\xi|^2 |\hat u_\delta(t,\xi)|^2 \, \d \xi \\
& \qquad\quad \leq C \delta^2 R^2 \int_{\R^d} |\xi|^2 |\hat u_0(\xi)|^2 \, \d \xi + C \delta^2 R^2 \int_{\R^d} \frac{|\xi|^2}{\omega_\delta^2(\xi)} |\hat v_0(\xi)|^2 \, \d \xi \, .
\end{split}
\]
By Lemma~\ref{lem:omega_delta_low_high_frequencies}, for~$R$ small enough we have that 
\[
\omega_\delta^2(\xi) \geq \frac{\gamma^2}{2} |\xi|^2 \, , \quad \text{for } |\xi| \leq R \, .
\]
It follows that 
\[
\begin{split}
& \Big| [u_\delta(t,\cdot)]_{\W_\delta^\alpha}^2 - \gamma^2 \|\nabla u_\delta(t,\cdot)\|_{L^2}^2 \Big| \leq C \delta^2 R^2 \left(\frac{1}{2} \int_{\R^d} |\hat v_0(\xi)|^2 \, \d \xi +  \int_{\R^d} |\xi|^2 |\hat u_0(\xi)|^2 \, \d \xi \right) \\
&\qquad \leq C \delta^2 R^2 \left(\frac{1}{2} \int_{\R^d} |v_0(x)|^2 \, \d x +  \int_{\R^d} |\nabla u_0(x)|^2 \, \d x \right) = C \delta^2 R^2 E^{\gamma^2\Delta}[u_\delta](0) \, .
\end{split} 
\]
Choosing suitably~$\delta$ and~$R$, we obtain the desired result.
\end{proof}

\begin{acknowledgements}
   G.\ M.\ Coclite, F.\ Maddalena, and G.\ Orlando are members of Gruppo Nazionale per l'Analisi Matematica, la Probabilit\`a e le loro Applicazioni (GNAMPA) of the Istituto Nazionale di Alta Matematica (INdAM).
 
   G.\ M.\ Coclite, F.\ Maddalena, and G.\ Orlando have been partially supported by the Research Project of National Relevance ``Evolution problems involving interacting scales'' granted by the Italian Ministry of Education, University and Research (MUR Prin 2022, project code 2022M9BKBC, Grant No. CUP D53D23005880006).
 
   G.\ M.\ Coclite, F.\ Maddalena, and G.\ Orlando acknowledge financial support under the National Recovery and Resilience Plan (NRRP) funded by the European Union - NextGenerationEU -  
Project Title ``Mathematical Modeling of Biodiversity in the Mediterranean sea: from bacteria to predators, from meadows to currents'' - project code P202254HT8 - CUP B53D23027760001 -  
Grant Assignment Decree No. 1379 adopted on 01/09/2023 by the Italian Ministry of University and Research (MUR).

G.\ M.\ Coclite, F.\ Maddalena, and G.\ Orlando  were partially supported by the Italian Ministry of University and Research under the Programme ``Department of Excellence'' Legge 232/2016 (Grant No. CUP - D93C23000100001).

   S.\ Dipierro acknowledges financial support
by the Australian Research Council Future Fellowship FT230100333 New
perspectives on nonlocal equations.

E.\ Valdinoci acknowledges financial support
by the Australian Laureate Fellowship FL190100081 Minimal surfaces,
free boundaries and partial differential equations.

\end{acknowledgements}

\printbibliography

\end{document}